\documentclass[leqno,10pt]{amsart}
\usepackage{amssymb,amsthm,wasysym}
\usepackage[latin2]{inputenc}

\theoremstyle{plain}

\newtheorem{thm}{Theorem}[section]
\newtheorem{cor}[thm]{Corollary}
\newtheorem{pro}[thm]{Proposition}

\newtheorem{lem}[thm]{Lemma}
\newtheorem{rem}[thm]{Remark}
\newtheorem*{ex*}{Example}

\numberwithin{equation}{section}

\newcommand{\B}{\mathbb{B}}
\newcommand{\N}{\mathbb{N}}
\newcommand{\R}{\mathbb{R}_{+}^{d}}
\def\a{\alpha}
\def\z{\zeta}
\def\ll{\lambda}
\def\t{\theta}
\def\v{\varphi}
\def\b{(\z,q_{\pm})}
\def\c{(\z,q_{\pm}(\t,y,s))}

\def\bbb{(\z,q_{\pm}(x,y,s))}
\def\ccc{(\z,q_{\pm}(x',y,s))}
\def\ddd{(\z,q_{\pm}(x,y',s))}

\DeclareMathOperator{\e}{{\mathbb{E}\textrm{xp}}}
\DeclareMathOperator{\lo}{{\mathbb{L}\textrm{og}}}
\DeclareMathOperator{\domain}{Dom}
\DeclareMathOperator{\supp}{supp}

\begin{document}

\title[g-functions based on Laguerre semigroups]{Littlewood-Paley-Stein type square functions based on Laguerre semigroups}
\author[]{Tomasz Szarek}
\address{Tomasz Szarek,     \newline
      ul{.} W{.} Rutkiewicz 29\slash 43,
      PL-50--571 Wroc\l{}aw, Poland        \vspace{10pt}}

\email{szarektomaszz@gmail.com}

\begin{abstract}
We investigate $g$-functions based on semigroups related to multi-dimensional Laguerre function expansions of convolution type. We prove that these operators can be viewed as Calder\'on-Zygmund operators in the sense of the underlying space of homogeneous type, hence their mapping properties follow from the general theory.
\end{abstract}

\maketitle

\section{Introduction}
\setcounter{equation}{0}
Square functions are one of the most fundamental concepts in harmonic analysis. Their study began in the twenties of the last century and since that time they were investigated by many authors in different contexts and in variety of forms, see \cite{Ste2} for a historical survey and also for a partial account of more recent developments. From the present perspective, square functions are important tools having several significant applications, for instance in the study of $H^p$ spaces, non-tangential convergence of Fatou type and the boundedness of Riesz transforms and multipliers; see \cite{Ste,Ste3,Ste4}. The aim of this paper is to study $g$-functions related to multi-dimensional expansions into Laguerre functions of convolution type. Our motivation comes not only from the fact that these objects are interesting on their own, but also from the possibility of their potential applications in further research concerning Laguerre expansions. The results we prove fit into the line of investigations conducted in the recent years and treating Littlewood-Paley theory for discrete and continuous orthogonal expansions, see for example \cite{BFS,BHNV,BMT,HRST,NoSj,StTo2} and also references therein. Some earlier results concerning $g$-functions in various Laguerre contexts, but different from ours, can be found in \cite{BFRST,BMR,GIT,No2,T}, among others. Certain one-dimensional results related to our setting are contained in \cite[Section 2]{StTr}.

We shall work on the space $\R=(0,\infty)^{d}$, $d\geq 1$, equipped with the measure
$$
d\mu_{\a}(x)=x_{1}^{2\a_{1}+1}\cdot\ldots\cdot x_{d}^{2\a_{d}+1}\,dx,\qquad x=(x_{1},\ldots,x_{d})\in\R,
$$
and with the Euclidean norm $|\cdot|$. The multi-index $\a=(\a_{1},\ldots,\a_{d})$ will always be assumed to belong to $(-1,\infty)^{d}$. Since $\mu_{\a}$ possesses the doubling property, the triple $(\R,d\mu_{\a},|\cdot|)$ forms the space of homogeneous type in the sense of Coifman and Weiss \cite{CW}. The Laguerre differential operator
\begin{align*}
L_{\a}=
-\Delta+|x|^{2}-\sum_{i=1}^{d} \frac{2\a_{i}+1}{x_{i}}\frac{\partial}{\partial x_{i}}.
\end{align*} 
is formally symmetric in $L^2(d\mu_{\a})$ and will play the role of Laplacian in the present setting. Appropriate partial derivatives $\delta_{j}$, $j=1,\ldots,d,$ related to $L_{\a}$ are obtained from the decomposition
\begin{align*}
L_{\a}=2|\a|+2d+\sum_{j=1}^{d}\delta_{j}^{*}\delta_{j},
\end{align*}
where
\begin{align*}
\delta_{j}=
\frac{\partial}{\partial x_{j}}+x_{j},\qquad 
\delta_{j}^{*}
=-\frac{\partial}{\partial x_{j}}+x_{j}-\frac{2\a_{j}+1}{x_{j}},\qquad j=1,\ldots,d,
\end{align*}
$\delta_{j}^{*}$ being the formal adjoint of $\delta_{j}$ in $L^2(d\mu_{\a})$; see \cite{NS1}, or \cite{NS4} for a more general background. It is well known that the system $\{\ell_{k}^{\a}\}$ of Laguerre functions of convolution type (see Section \ref{pre} for the definition) forms an orthonormal and complete in $L^{2}(d\mu_{\a})$ set of eigenfunctions of $L_{\a}$.

We shall consider vertical (involving 'time' derivative $\partial_{t}$) and horizontal (involving 'space' derivatives $\delta_{j}$ or $\delta_{j}^{*}$) $g$-functions based on the heat and Poisson semigroups generated by $L_{\a}$ and on 'modified' versions of these semigroups (such 'modifications' emerge naturally in the associated conjugacy theory and are also of interest, cf. \cite{NS4,NS1}); see Section \ref{pre} for a complete list. Our main result, Theorem \ref{gfun} below, says that under the slight restriction $\a\in[-1\slash 2,\infty)^{d}$ all the $g$-functions can be viewed as vector-valued Calder\'on-Zygmund operators in the sense of the space of homogeneous type $(\R,d\mu_{\a},|\cdot|)$. Consequences of this, including $L^{p}$ mapping properties, are then delivered by the general theory. The main difficulty connected with the Calder\'on-Zygmund approach is to show relevant kernel estimates. Here we use a convenient technique having roots in Sasso's paper \cite{Sa} and developed by Nowak and Stempak in \cite{NS1}. This method is based on Schl\"afli's formula that allows to handle modified Bessel functions entering integral kernels of the Laguerre semigroups. It is remarkable that the same technique is also well suited to $g$-functions of higher order and, if necessary, can be used to prove that they can be viewed as vector-valued Calder\'on-Zygmund operators. The related analysis, however, is lengthy and rather sophisticated, and thus beyond the scope of this article.

The paper is organized as follows. Section \ref{pre} contains the setup, definitions of $g$-functions, statements of the main results and the accompanying comments and remarks. In Section \ref{o+s} the $g$-functions are proved to be $L^2$ bounded and associated, in the Calder\'on-Zygmund theory sense, with the relevant kernels. Finally, Section \ref{ker} is devoted to the proofs of all necessary kernel estimates. This is the largest and most technical part of this work.

Throughout the paper we use a standard notation with essentially all symbols referring to the homogeneous space $(\R,d\mu_{\a},|\cdot|)$. Thus $\Delta$ and $\nabla$ denote the Laplacian and gradient, respectively, restricted to $\R$. Further, $L^{p}(wd\mu_{\a})$ stands for the weighted $L^{p}(d\mu_{\a})$ space, $w$ being a nonnegative weight on $\R$; we write simply $L^{p}(d\mu_{\a})$ if $w\equiv 1$. By $\langle f,g\rangle_{d\mu_{\a}}$ we mean $\int_{\R}f(x)\overline{g(x)}\,d\mu_{\a}(x)$ whenever the integral makes sense. For  $1\leq p<\infty$ we denote by $A_{p}^{\a}$ the Muckenhoupt class of $A_{p}$ weights associated to the space $(\R,d\mu_{\a},|\cdot|)$ (see \cite[p.\,645]{NS1} for a precise description). While writing estimates we will frequently use the notation $X\lesssim Y$ to indicate that $X\leq CY$ with a positive constant $C$ independent of significant quantities. We shall write $X\simeq Y$ when $X\lesssim Y$ and $Y\lesssim X$.

\textbf{Acknowledgments.} The author would like to thank Dr. Adam Nowak for suggesting the topic and constant support during preparation of the paper.

\section{Preliminaries and main results}\label{pre}
\setcounter{equation}{0}

Let $k=(k_1,\ldots,k_d) \in \N^d$, $\N=\{0,1,\dots\}$,  and
$\alpha = (\alpha _1, \ldots , \alpha _d) \in (-1,\infty)^{d}$
be multi-indices.
The Laguerre function $\ell^{\alpha}_k$ is defined on $\R$ as the tensor product
$$
\ell_{k}^{\alpha}(x) = \ell _{k_1}^{\alpha _1}(x_1) \cdot \ldots \cdot
\ell_{k_d}^{\alpha _d}(x_d), \qquad x = (x_1, \ldots ,x_d)\in \R,
$$
where $\ell_{k_i}^{\alpha _i}$ are the one-dimensional Laguerre functions
$$
\ell_{k_i}^{\alpha _i}(x_i) =
\left(\frac{2 \Gamma (k_i+1)}{\Gamma(k_i+\alpha _i +1)}\right)
^{1 \slash 2} L_{k_i}^{\alpha _i} (x_i^{2})e^{-{x_i^2}/{2}}, \quad \quad x_i > 0,
\quad i = 1,\ldots , d;
$$
here $L^\alpha_k$ denotes the Laguerre polynomial 
of degree $k$ and order $\alpha,$ see \cite[p.\,76]{Leb}.
Each $\ell_k^{\alpha}$ is an eigenfunction of the Laguerre operator $L_\alpha$
with the corresponding eigenvalue $4|k|+2|\alpha |+2d$, that is
$$
L_\alpha\ell_k^{\alpha}=(4|k|+2|\alpha |+2d)\ell_k^{\alpha};
$$
by $|\alpha|$ and $|k|$ we denote $|\alpha|=\alpha_1+\ldots+\alpha_d$
(thus $|\alpha|$ may be negative) and the length $|k|=k_1+\ldots+k_d$.
Furthermore, the system $\{ \ell_{k}^{\alpha} : k \in \N^d \}$ is an orthonormal
basis in $L^2(d\mu_{\alpha})$ and consequently each of the 'differentiated' (see \eqref{opdel} below) systems $\{ x_{j}\ell_{k}^{\alpha+e_{j}} : k \in \N^d \}$, $j=1,\ldots,d$, is an orthonormal basis in $L^2(d\mu_{\alpha})$; here $e_{j}$ is the $j$th coordinate vector.

The operator $L_\alpha$, considered initially on $C^{\infty}_c(\R)\subset L^{2}(d\mu_{\a})$, is symmetric and positive. We take into account a natural self-adjoint extension $\mathcal{L}_{\a}$ of $L_{\a}$ (see \cite[p.\,646]{NS1}), whose spectral decomposition is given by
$$
\mathcal{L}_\alpha f=\sum_{n=0}^\infty\ll_{n}^{\a}\,\mathcal{P}_n^\alpha f, \qquad\ll_{n}^{\a}=4n+2|\a|+2d,
$$
on the domain $\domain \mathcal{L}_{\a}$ consisting of all functions $f\in L^{2}(d\mu_{\a})$ for which the defining series converges in $L^{2}(d\mu_{\a})$; here $\mathcal{P}_n^\alpha $ are the spectral projections
\begin{equation*}
\mathcal{P}_n^\alpha f=\sum_{|k|=n}\langle f,\ell_k^{\alpha}
  \rangle_{d\mu_{\alpha}}\ell_k^{\alpha}.
\end{equation*}
The semigroup $T_t^\alpha = \exp(-t\mathcal{L}_{\alpha})$, $t \ge 0$, generated by
$\mathcal{L}_{\alpha}$ is a strongly continuous semigroup of contractions in
$L^2(d\mu_{\alpha})$. By the spectral theorem,
\begin{equation} \label{pc}
T_t^\alpha f=\sum_{n=0}^\infty e^{-t\ll_{n}^{\a}}\mathcal{P}^\alpha_nf, \qquad f\in L^2(d\mu_{\alpha}).
\end{equation}
We have the integral representation 
\begin{equation*} 
  T_t^\alpha f(x)=\int_{\R} G_t^\alpha(x,y)f(y)\,d\mu_{\alpha}(y),
  \qquad x\in \R,
\end{equation*}
where the Laguerre heat kernel is given by
\begin{equation*} 
G^\alpha_t(x,y)=\sum_{n=0}^\infty e^{-t\ll_{n}^{\a}} \sum_{|k|=n}
\ell_k^\alpha(x)\ell_k^\alpha(y).
\end{equation*}
This series can be summed (see \cite[(4.17.6)]{Leb}) and the resulting formula is
$$
G^\alpha_t(x,y)=(\sinh 2t)^{-d}\exp\Big({-\frac{1}{2} \coth(2t)\big(|x|^{2}+|y|^{2}\big)}\Big)
\prod^{d}_{i=1} (x_i y_i)^{-\alpha_i} I_{\alpha_i}\left(\frac{x_i y_i}{\sinh 2t}\right),
$$
with $I_\nu$ denoting the modified Bessel function of the first kind and order $\nu$, cf. \cite[Section 5]{Leb}.

We consider also the operators
$$
\mathcal{M}_j^\alpha
f=\sum_{n=1}^\infty\ll_{n}^{\a}\,\mathcal{P}_n^{\alpha,j} f,\qquad j=1,\ldots,d,
$$
with domains $\domain \mathcal{M}_{j}^{\a}$ consisting of all functions such that the defining series converge in $L^{2}(d\mu_{\a})$; here the spectral projections $\mathcal{P}_n^{\alpha,j}$ are given by
$$
\mathcal{P}_n^{\alpha,j} f=\sum_{|k|=n}\langle f,x_j\ell
_{k-e_j}^{\alpha+e_j}
  \rangle_{d\mu_{\alpha}}\, x_j\ell_{k-e_j}^{\alpha+e_j},
$$ 
where, by convention, $\ell_{k-e_{j}}^{\a+e_{j}}=0$ if $k_{j}-1<0$ (this convention will also be used in the sequel). 
According to \cite[Section 4]{NS1}, $\mathcal{M}_j^\alpha$ are self-adjoint extensions of the differential operators 
$$
M^{\alpha}_j =L_\alpha + \frac{2\alpha_j+1}{x^{2}_j} + 2, \quad \quad
j=1,\ldots, d.
$$
These perturbations of the Laguerre operator emerge naturally in the conjugacy theory for Laguerre expansions, see \cite[Section 5]{NS4} for a general background.
The 'modified' Laguerre semigroups $\widetilde T_{t}^{\a,j} = \exp(-t\mathcal{M}_{j}^{\alpha})$, $t \ge 0$, generated by
$\mathcal{M}_{j}^{\alpha}$ are given on $L^2(d\mu_{\a})$ by
\begin{align}\label{tfal}
\widetilde T_{t}^{\a,j} f=\sum_{n=0}^{\infty}
e^{-t\ll_{n}^{\a}} 
    \mathcal{P}^{\alpha,j}_n f,\qquad j=1,\ldots,d.
\end{align}
The integral representation of $\widetilde{T}^{\alpha,j}_t$ is
$$
\widetilde{T}^{\alpha,j}_t f(x) = \int_{\R}
\widetilde{G}^{\alpha,j}_t(x,y) f(y) \, d\mu_{\alpha}(y), \qquad x \in \R,
    \quad t>0,
$$
with (see \cite[p.\,662]{NS1})
\begin {equation}\label{gfal}
\widetilde{G}^{\alpha,j}_t(x,y)=\sum_{n=0}^\infty e^{-t\ll_{n}^{\a}} \sum_{|k|=n}
x_{j}y_{j}\ell_{k-e_{j}}^{\a+e_{j}}(x)\ell_{k-e_{j}}^{\a+e_{j}}(y)
=
e^{-2t}x_{j}y_{j} G^{\alpha+e_j}_t(x,y).
\end {equation}

The Poisson semigroups $\{P^{\alpha}_t\}_{t\ge 0}$, $\{\widetilde P_{t}^{\a,j}\}_{t\ge 0}$, $j=1,\ldots,d$, associated with $\mathcal{L}_{\alpha}$ and $\mathcal{M}_{j}^{\alpha}$, respectively, are in view of the spectral theorem given by
\begin{align}
P^{\alpha}_t f 
=&
\,e^{-t\sqrt{\mathcal{L}_{\alpha}}} f 
=
\sum_{n=0}^{\infty} e^{-t\sqrt{\ll_{n}^{\a}}}
    \mathcal{P}^{\alpha}_n f, \qquad f \in L^2(d\mu_{\alpha}),\\
\widetilde P^{\alpha,j}_t f 
=&
\,e^{-t\sqrt{\mathcal{M}_{j}^{\alpha}}} f 
=
\sum_{n=0}^{\infty} e^{-t\sqrt{\ll_{n}^{\a}}}
    \mathcal{P}^{\alpha,j}_n f, \qquad f \in L^2(d\mu_{\alpha}),\quad  j=1,\ldots,d.\label{pp}\\\nonumber
\end{align}
An important connection between heat-diffusion and Poisson semigroups is established by the subordination principle,
\begin{align*}
P^{\alpha}_t f (x) =& \frac{1}{\sqrt{\pi}} \int_0^{\infty} \frac{e^{-u}}{\sqrt{u}} T^{\alpha}_{t^2\slash (4u)}
    f(x) \, du, \qquad x \in \R,\\
\widetilde P^{\alpha,j}_t f (x) =& \frac{1}{\sqrt{\pi}} \int_0^{\infty} \frac{e^{-u}}{\sqrt{u}} \widetilde T^{\alpha,j}_{t^2\slash (4u)}
    f(x) \, du, \qquad x \in \R,\quad j=1,\ldots,d.
\end{align*}

We consider the following vertical and horizontal $g$-functions based on the Laguerre heat semigroup and its 'modifications',
\begin{displaymath}
\begin {array}{lll}
g_{V,T}(f)(x)&=
\big\|\partial_{t}T_{t}^{\a}f(x)\big\|_{L^{2}(tdt)},&\\
g_{H,T}^{i}(f)(x)&=
\big\|\delta_{i}T_{t}^{\a}f(x)\big\|_{L^{2}(dt)},\qquad&
i=1,\ldots,d,\\
g_{V,\widetilde T}^{j}(f)(x)&=
\big\|\partial_{t}\widetilde T_{t}^{\a,j}f(x)\big\|_{L^{2}(tdt)}
,\qquad& j=1,\ldots,d,\\
g_{H,\widetilde T}^{j,i}(f)(x)&=
\big\|\delta_{i}\widetilde T_{t}^{\a,j}f(x)\big\|_{L^{2}(dt)}
,\qquad&
i,j=1,\ldots,d,\quad j\ne i,\\
g_{H,\widetilde T}^{j,j}(f)(x)&=
\big\|\delta_{j}^{*}\widetilde T_{t}^{\a,j}f(x)\big\|_{L^{2}(dt)}
,\qquad&
j=1,\ldots,d,\\
\end {array}
\end{displaymath}
and their analogues involving the Poisson semigroups,
\begin{displaymath}
\begin {array}{lll}
g_{V,P}(f)(x)&=
\big\|\partial_{t}P_{t}^{\a}f(x)\big\|_{L^{2}(tdt)},
&\\
g_{H,P}^{i}(f)(x)&=
\big\|\delta_{i}P_{t}^{\a}f(x)\big\|_{L^{2}(tdt)},&
\qquad i=1,\ldots,d,\\
g_{V,\widetilde P}^{j}(f)(x)&=
\big\|\partial_{t}\widetilde P_{t}^{\a,j}f(x)\big\|_{L^{2}(tdt)}
,&
\qquad j=1,\ldots,d,\\
g_{H,\widetilde P}^{j,i}(f)(x)&=
\big\|\delta_{i}\widetilde P_{t}^{\a,j}f(x)\big\|_{L^{2}(tdt)}
,&
\qquad i,j=1,\ldots,d,\quad j\ne i,\\
g_{H,\widetilde P}^{j,j}(f)(x)&=
\big\|\delta_{j}^{*}\widetilde P_{t}^{\a,j}f(x)\big\|_{L^{2}(tdt)}
,&
\qquad j=1,\ldots,d.\\
\end {array}
\end{displaymath}
Clearly, the square functions just listed are nonlinear, but by the well-known trick they can be identified with linear operators. For example $g_{V,T}$ can be viewed as the vector-valued linear operator $f(x)\mapsto \{\partial_{t}T_{t}^{\a}f(x)\}_{t>0}$ which maps into functions of $x$ having values in $L^2(tdt)$.

The main result of the paper reads as follows.
\begin{thm}\label{gfun}
Assume that $\alpha\in [-1\slash 2,\infty)^{d}$. Then each of the square functions listed above, viewed as a vector-valued operator related to either $L^{2}(dt)$ (the cases of $g_{H,T}^{i}$ and $g_{H,\widetilde T}^{j,i}$) or to $L^{2}(tdt)$ (the remaining cases), is a Calder\'on-Zygmund operator in the sense of the space of homogeneous type $(\R,d\mu_{\alpha},|\cdot|)$. 
\end{thm}
Let $\B$ be a Banach space and let $K(x,y)$ be a kernel defined on $\R\times\R\backslash\{(x,y) : x=y\}$ and taking values in $\B$. We say that $K(x,y)$ is a standard kernel in the sense of the space of homogeneous type $(\R,d\mu_{\a},|\cdot|)$ if it satisfies the growth estimate
\begin{align}\label{gr}
\|K(x,y)\|_{\B}
&\lesssim
\frac{1}{\mu_{\alpha}(B(x,|y-x|))},\\\nonumber
\end{align}
and the smoothness estimates
\begin{align}\label{sm1}
\|K(x,y)-K(x',y)\|_{\B}
&\lesssim
\frac{|x-x'|}{|x-y|} \;
\frac{1}{\mu_{\alpha}(B(x,|y-x|))},\qquad |x-y|>2|x-x'|,\\\label{sm2}
\|K(x,y)-K(x,y')\|_{\B}
&\lesssim
\frac{|y-y'|}{|x-y|} \;
\frac{1}{\mu_{\alpha}(B(x,|y-x|))},\qquad |x-y|>2|y-y'|.\\\nonumber
\end{align}
Notice that here, in view of the doubling property of $\mu_{\a}$, in any occurrence the ball $B(x,|y-x|)$ can be replaced by $B(y,|y-x|)$.

A linear operator $T$ assigning to each $f\in L^2(d\mu_{\a})$ a measurable $\B$-valued function $Tf$ on $\R$ is a (vector-valued) Calder\'on-Zygmund operator in the sense of the space $(\R,d\mu_{\alpha},|\cdot|)$ if
\begin{enumerate}
    \item $T$ is bounded from $L^2(d\mu_{\a})$ to $L^2_{\B}(d\mu_{\a})$,
  \item there exists a standard $\B$-valued kernel $K(x,y)$ such that
\begin{align*}
Tf(x)=\int_{\R}K(x,y)f(y)\,d\mu_{\a}(y),\qquad \textrm{a.e.}\,\,\, x\notin \supp f,
\end{align*} 
for every $f\in L^2(\R,d\mu_{\a})$ vanishing outside a compact set contained in $\R$ (we write shortly $T\sim K(x,y)$ for this kind of association).
\end{enumerate}
Here integration of $\B$-valued functions is understood in Bochner's sense, and $L_{\B}^{2}$ is the Bochner-Lebesgue space of all $\B$-valued $d\mu_{\a}$-square integrable functions on $\R$.
It is well known that a large part of the classical theory of Calder\'on-Zygmund operators remains valid, with appropriate adjustments, when the underlying space is of homogeneous type and the associated kernels are vector-valued, see the comments in \cite[p.\,649]{NS1} and references given there.

The proof of Theorem \ref{gfun} splits naturally into showing the following three results.

\begin{pro}\label{preogr}
Let $\a\in[-1\slash 2,\infty)^{d}$. The $g$-functions listed above are bounded on $L^2(d\mu_{\a})$. Consequently, each of the $g$-functions, viewed as a vector-valued operator, is bounded from $L^2(d\mu_{\a})$ to $L^2_{\B}(d\mu_{\a})$, where $\B=L^2(dt)$ in the cases of $g_{H,T}^{i}$ and $g_{H,\widetilde T}^{j,i}$, and $\B=L^2(tdt)$ in the remaining cases.
\end{pro}

\begin{pro}\label{prestowP}
Let $\alpha\in [-1\slash 2,\infty)^{d}$. Then the square functions under consideration, viewed as vector-valued operators related to $\B=L^2(dt)$ (the cases of $g_{H,T}^{i}$ and $g_{H,\widetilde T}^{j,i}$) or $\B=L^2(tdt)$ (the remaining cases), are associated with the following kernels:
\begin{displaymath}
\begin {array}{lll}
g_{V,T}\sim\big\{\partial_{t}G_{t}^{\a}(x,y)\big\}_{t>0}
,\qquad
&g_{V,P}\sim\big\{\partial_{t}P_{t}^{\a}(x,y)\big\}_{t>0}
,
&\\
g_{H,T}^{i}\sim\big\{\delta_{i}G_{t}^{\a}(x,y)\big\}_{t>0}
,\qquad
&g_{H,P}^{i}\sim\big\{\delta_{i}P_{t}^{\a}(x,y)\big\}_{t>0}
,
&\\
g_{V,\widetilde T}^{j}\sim\big\{\partial_{t}\widetilde G_{t}^{\a,j}(x,y)\big\}_{t>0}
,\qquad
&g_{V,\widetilde P}^{j}\sim\big\{\partial_{t}\widetilde P_{t}^{\a,j}(x,y)\big\}_{t>0}
,
&\qquad j=1,\ldots,d,\\
g_{H,\widetilde T}^{j,i}\sim\big\{\delta_{i}\widetilde G_{t}^{\a,j}(x,y)\big\}_{t>0}
,\qquad
&g_{H,\widetilde P}^{j,i}\sim\big\{\delta_{i}\widetilde P_{t}^{\a,j}(x,y)\big\}_{t>0}
,
&\qquad i,j=1,\ldots,d,\quad j\ne i,\\
g_{H,\widetilde T}^{j,j}\sim\big\{\delta_{j}^{*}\widetilde G_{t}^{\a,j}(x,y)\big\}_{t>0}
,\qquad
&g_{H,\widetilde P}^{j,j}\sim\big\{\delta_{j}^{*}\widetilde P_{t}^{\a,j}(x,y)\big\}_{t>0}
,
&\qquad j=1,\ldots,d.\\
\end {array}
\end{displaymath}
\end{pro}

\begin{thm}\label{preg}
Assume that $\a\in[-1\slash 2,\infty)^{d}$. Let $K(x,y)$ be any of the vector-valued kernels listed in Proposition \ref{prestowP}. Then $K(x,y)$ satisfies the standard estimates \eqref{gr}, \eqref{sm1}, \eqref{sm2}, with $\B=L^2(dt)$ in the cases of $\{\delta_{i}G_{t}^{\a}(x,y)\}_{t>0}$, $\{\delta_{i}\widetilde G_{t}^{\a,j}(x,y)\}_{t>0}$, $i\ne j$ and $\{\delta_{j}^{*}\widetilde G_{t}^{\a,j}(x,y)\}_{t>0}$, and $\B=L^2(tdt)$ in the remaining cases.
\end{thm}

Proofs of Propositions \ref{preogr} and \ref{prestowP} are given in Section \ref{o+s} (in fact we show somewhat stronger result than Proposition \ref{preogr}). The proof of Theorem \ref{preg} is the most technical part of the paper and is located in Section \ref{ker}.
The restriction  $\a\in[-1\slash 2,\infty)^{d}$ appearing in Theorem \ref{preg} and consequently in Proposition \ref{prestowP} and in Theorem \ref{gfun} is imposed by the applied method of proving standard estimates; see \cite[p.\,666]{NS1} for more comments.

Next, we observe that the $g$-functions under consideration can be naturally defined pointwise, by the same formulas, for general functions $f$ from weighted spaces $L^{p}(wd\mu_{\a})$, $1\leq p<\infty$, $w\in A_{p}^{\a}$. Indeed, the estimates (cf. \cite[(2.7), (2.8)]{NS1})
\begin{align}\label{lka}
|\langle f,\ell_{k}^{\a}\rangle|
\lesssim
\big(|k|+1\big)^{c_{d,\a}}\|f\|_{L^{p}(wd\mu_{\a})},\qquad
|\ell_{k}^{\a}(x)|
\lesssim
\big(|k|+1\big)^{c_{d,\a}},\qquad k\in \N^{d},\quad x\in\R,
\end{align}
and parallel estimates for the 'differentiated' Laguerre systems $\{x_{j}\ell_{k}^{\a+e_{j}}\}$, $j=1,\ldots,d$, allow to verify that the spectral series and integral representations of the relevant semigroups converge for such general $f$ producing smooth functions of $(t,x)\in\mathbb{R}_{+}\times\R$; see \cite[Section 4]{No1} and comments in \cite[Section 2]{NS1}. This provides the relevant extensions of the semigroups to the weighted $L^{p}$ spaces. As a consequence of Theorem \ref{gfun} we state the following.

\begin{cor}\label{ext}
Let $\a\in[-1\slash 2,\infty)^{d}$. Then each of the square functions listed above is bounded on $L^{p}(w d\mu_{\a})$, $w\in A_{p}^{\a}$, $1<p<\infty$, and from $L^1(w d\mu_{\a})$ to weak $L^{1}(w d\mu_{\a})$, $w\in A_{1}^{\a}$.
\end{cor}
Further consequences of Theorem \ref{gfun} can be derived from the general theory of Calder\'on-Zygmund operators, see \cite[Theorem 1.1]{BFS}. We leave details to interested readers.
\begin {proof}[Proof of Corollary \ref{ext}.]
We give a detailed reasoning only for $g_{V,T}$, adapting suitably the arguments used in the proof of \cite[Theorem 2.2]{StTo2}. The remaining cases are proved similarly.

Fix $1\leq p<\infty$ and $w\in A_{p}^{\a}$. Let $G_{V,T}\colon L^{p}(wd\mu_{\a})\cap L^{2}(d\mu_{\a})\mapsto L^{2}_{L^2(tdt)}(wd\mu_{\a})$ be the vector-valued operator emerging from $g_{V,T}$.
By virtue of Theorem \ref{gfun} and the general theory the operator $G_{V,T}$ extends uniquely to a bounded linear operator $H$ from $L^{p}(wd\mu_{\a})$ to $L^{p}_{L^2(tdt)}(wd\mu_{\a})$ if $p>1$, or from $L^{1}(wd\mu_{\a})$ to the weak $L^{1}_{L^2(tdt)}(wd\mu_{\a})$ if $p=1$. This reduces the proof to showing that $Hf(x)=\{\partial_{t}T_{t}^{\a}f(x)\}_{t>0}$ for each $f\in L^{p}(wd\mu_{\a})$. For a given $f\in L^{p}(wd\mu_{\a})$ we choose a sequence $f_{n}\in L^{p}(wd\mu_{\a})\cap L^{2}(d\mu_{\a})$ such that $f_{n}\to f$ in $L^{p}(wd\mu_{\a})$. Let $g_{n}(t,x)=\partial_{t}T_{t}^{\a}f_{n}(x)$ and $g(t,x)=\partial_{t}T_{t}^{\a}f(x)$. Using \eqref{lka} it is not hard to check that for every fixed $t>0$ we have $g_{n}(t,x)\to g(t,x)$ pointwise. Thus assuming that $p>1$ (the case $p=1$ is treated similarly) and applying Fatou's lemma twice we obtain 
\begin{align*}
\int_{\R}&\bigg(\int_{0}^{\infty}\big|g(t,x)-Hf(t,x)\big|^{2}t\,dt\bigg)
^{p\slash 2}w(x)\,d\mu_{\a}(x)\\
=&
\int_{\R}\bigg(\int_{0}^{\infty}\liminf_{n\to\infty}\big|g_{n}(t,x)-Hf(t,x)\big|^{2}t\,dt\bigg)
^{p\slash 2}w(x)\,d\mu_{\a}(x)\\
\leq&
\liminf_{n\to\infty}\int_{\R}\bigg(\int_{0}^{\infty}\big|g_{n}(t,x)-Hf(t,x)\big|^{2}t\,dt\bigg)
^{p\slash 2}w(x)\,d\mu_{\a}(x).
\end{align*} 
The last expression is equal $0$ since $Hf$ is a limit of $G_{V,T}(f_{n})=\{g_{n}(t,x)\}_{t>0}$ in $L^{p}_{L^2(tdt)}(wd\mu_{\a})$. It follows that $g(t,x)=Hf(t,x)$ in $L^{p}_{L^2(tdt)}(wd\mu_{\a})$, which ends the proof.
\end {proof}
We conclude this section with various comments and remarks related to the main result.

\begin{rem}\label{ogrd}
Corollary \ref{ext} and standard arguments (see for instance \cite[p.\,85]{Ste}) allow to obtain also lower $L^{p}$ estimates for some of the considered $g$-functions. Moreover, this concerns the weighted setting with $A_{p}^{\a}$ weights admitted (cf. \cite[p.\,207]{StTr} for the special case of power $A_{p}^{\a}$ weights). More precisely, if $g$ stands for any of the vertical $g$-functions listed above, then under the assumption $\a\in[-1\slash 2,\infty)^{d}$ and for $1<p<\infty$, $w\in A_{p}^{\a}$ we have
\begin{align*}
\|f\|_{L^{p}(wd\mu_{\a})}
\lesssim
\|g(f)\|_{L^{p}(wd\mu_{\a})},\qquad f\in L^{p}(wd\mu_{\a}).
\end{align*}
A crucial fact in proving this is that $g$ is essentially an isometry on $L^2(d\mu_{\a})$. This is the case of all the vertical $g$-functions (see Proposition \ref{ogrv} below), but not the case of our horizontal $g$-functions, as can be seen in Section \ref{o+s}. Another relevant fact is that a weight $w\in A_{p}^{\a}$ if and only if $w^{-p'\slash p}\in A_{p'}^{\a}$, with $p'$ being the adjoint of $p$, $1\slash p + 1\slash p'=1$.
\end{rem}

Corollary \ref{ext} and Remark \ref{ogrd}, specified to $g_{V,P}$, together extend \cite[Proposition 2.1]{StTr} in several directions, by admitting more general weights, wider range of $\a$, and a multi-dimensional setting.

\begin{rem}
Some special cases of Corollary \ref{ext} follow by a general theory. This concerns unweighted $L^{p}$, $p\ne 1$, estimates for the vertical $g$-functions. The general result to be invoked is a refinement of Stein's Littlewood-Paley theory for semigroups \cite{Ste3} due to Coifman, Rochberg and Weiss \cite{CRW}, see also \cite[Theorem 2]{M}. In our setting the consequence is the following. Let $\a\in[-1\slash 2,\infty)^{d}$ and $1<p<\infty$. Then, with $g$ being any of the considered vertical $g$-functions, we have 
\begin{align*}
\|g(f)\|_{L^{p}(d\mu_{\a})}\leq C
\|f\|_{L^{p}(d\mu_{\a})},\qquad f\in L^{p}(d\mu_{\a}).
\end{align*}
It is remarkable that the  constant $C$ here, in contrast to the result of Corollary \ref{ext}, depends neither on the dimension nor on the type multi-index $\a$. Similar dimension-free estimates are true for the higher-order vertical $g$-functions, see \cite[Theorem 2]{M}. 

To see that application of the result from \cite{CRW} is indeed possible one has to ensure that the semigroups in question are symmetric contraction semigroups. In our case only contractivity is non-trivial. It can be easily deduced from \cite[Section 2]{St} that the operators $T_{t}^{\a}$, $t>0$, are contractions on all $L^{p}$ spaces, $1\leq p\leq\infty$, and for all $\a\in(-1,\infty)^{d}$. Then an analogous conclusion is obtained for $\widetilde T_{t}^{\a,j}$, $t>0$, by means of the estimate (cf. \cite[Proposition 4.3]{NS1}) $\widetilde G_{t}^{\a,j}(x,y)\leq G_{t}^{\a}(x,y)$ and with the restriction $\a\in[-1\slash 2,\infty)^{d}$. Contractivity of the Poisson semigroups follows by the subordination principle.
\end{rem}

\begin{rem}
By the subordination principle, the $g$-functions based on Laguerre-Poisson semigroups can be controlled pointwise by their analogues based on Laguerre heat semigroups. This allows to pass with strong and weak type estimates in the direction 'heat$\to$Poisson'. However, the property of being a Calder\'on-Zygmund operator is more subtle and such a direct passage is not possible, and the subordination principle must be applied on the level of standard kernel estimates, see Sections \ref{k3} and \ref{k6}.
\end{rem}

\begin{rem}
It is not appropriate to exchange the roles of $\delta_{j}$ and $\delta_{j}^{*}$ in the definitions of the horizontal $g$-functions considered in this paper. Assume for simplicity that $d=1$ and focus for instance on $g_{H,T}^{1}$. Let $\widetilde g_{H,T}^{1}(f)(x)=\|\delta_{1}^{*}T_{t}^{\a}f(x)\|_{L^2(dt)}$ be the $g$-function arising by replacing $\delta_{1}$ with $\delta_{1}^{*}$ in the definition of $g_{H,T}^{1}$. A direct computation reveals that $\widetilde g_{H,T}^{1}(\ell_{0}^{\a})(x)=\frac{1}{\sqrt{4\a+4}}\big|2x-\frac{2\a+1}{x}\big|\ell_{0}^{\a}(x)$. Since $\ell_{0}^{\a}\in L^{p}(d\mu_{\a})$ for all $p\geq 1$ and $\widetilde g_{H,T}^{1}(\ell_{0}^{\a})\notin L^{p}(d\mu_{\a})$ for $p\geq 2\a+2$ and $\a\neq-1\slash 2$, we see that $\widetilde g_{H,T}^1$ cannot be viewed as a vector-valued Calder\'on-Zygmund operator. This counterexample can be easily generalized to the multi-dimensional situation and the cases of $g_{H,T}^{i}$, $g_{H,P}^{i}$ and $g_{H,\widetilde T}^{j,i}$, $g_{H,\widetilde P}^{j,i}$ for $i\ne j$, whereas the cases of $g_{H,\widetilde T}^{j,j}$, $g_{H,\widetilde P}^{j,j}$ are more subtle.
\end{rem}

\begin{ex*}
To illustrate the classic application of $g$-functions and to see that $g$-functions based on 'modified' semigroups are also of interest, we now briefly show how the weighted $L^{p}$-boundedness of Riesz-Laguerre transforms $R_{j}^{\a}$ proved in \cite[Theorem 3.4]{NS1} can be recovered from Corollary \ref{ext}. Given $w\in A_{p}^{\a}$ and $j=1,\ldots,d$, we have (cf. \cite[(4.7)]{NS1} and \cite[Proposition 4.5]{NS1})
\begin{align*}
\partial_{t}\widetilde P_{t}^{\a,j}R_{j}^{\a}f=
-\delta_{j}P_{t}^{\a}f,\qquad f\in L^{p}(wd\mu_{\a})\cap L^{2}(d\mu_{\a}).
\end{align*}
This implies
\begin{align*}
g_{V,\widetilde P}^{j}(R_{j}^{\a}f)=
g_{H,P}^{j}(f),\qquad f\in L^{p}(wd\mu_{\a})\cap L^{2}(d\mu_{\a}),
\end{align*}
so combining Corollary \ref{ext} (specified to $g_{H,P}^{j}$) with Remark \ref{ogrd} (specified to $g_{V,\widetilde P}^{j}$) leads to the estimate
\begin{align*}
\|R_{j}^{\a}f\|_{L^{p}(wd\mu_{\a})}
\lesssim
\|f\|_{L^{p}(wd\mu_{\a})},\qquad f\in L^{p}(wd\mu_{\a})\cap L^{2}(d\mu_{\a}),
\end{align*}
for $1<p<\infty$. Consequently, $R_{j}^{\a}$ extends uniquely to bounded linear operator on $L^{p}(wd\mu_{\a})$, $1<p<\infty$, $w\in A_{p}^{\a}$.
\end{ex*}

\section{$L^{2}$-Boundedness and Kernel associations}\label{o+s}
\setcounter{equation}{0}
In this section we analyze behavior of the square functions introduced in Section \ref{pre} on the Hilbert space $L^2(d\mu_{\a})$. It occurs that all the vertical $g$-functions are essentially isometries on $L^2(d\mu_{\a})$. However, this is not true for the horizontal $g$-functions, which are at most comparable in $L^2(d\mu_{\a})$ norm with original functions.
We also show that the $g$-functions, viewed as vector-valued operators, are associated with the relevant kernels.

The following two propositions together imply Proposition \ref{preogr}.

\begin{pro}\label{ogrv}
Assume that $\alpha\in (-1,\infty)^{d}$. Then, for $f\in L^2(d\mu_{\a})$, the vertical square functions satisfy
\begin{align*}
\|g_{V,T}(f)\|_{L^2(d\mu_{\a})}
=
\frac{1}{2}\|f\|_{L^2(d\mu_{\a})},\qquad&
\|g_{V,P}(f)\|_{L^2(d\mu_{\a})}
=
\frac{1}{2}\|f\|_{L^2(d\mu_{\a})},\\
\|g_{V,\widetilde T}^{j}(f)\|_{L^2(d\mu_{\a})}
=
\frac{1}{2}\|f\|_{L^2(d\mu_{\a})},\qquad&
\|g_{V,\widetilde P}^{j}(f)\|_{L^2(d\mu_{\a})}
=
\frac{1}{2}\|f\|_{L^2(d\mu_{\a})},\qquad j=1,\ldots,d.
\end{align*}
\end{pro}

\begin {proof}
Treatment of each of the square functions is based on similar arguments. Therefore we deal only with $g_{V,T}$, leaving the remaining cases to the reader.

Differentiating term by term the series in \eqref{pc} (this is legitimate in view of \eqref{lka}) we get
$$
\partial_{t}T_{t}^{\a}f(x)=
-\sum_{n=0}^{\infty}\ll_{n}^{\a}e^{-t\ll_{n}^{\a}}\sum_{|k|=n}\langle f,\ell_{k}^{\a}\rangle_{d\mu_{\a}}\ell_{k}^{\a}(x).
$$
Applying the Fubini-Tonelli theorem and Parseval's identity we obtain
\begin{align*}
\|g_{V,T}(f)\|_{L^2(d\mu_{\a})}^{2}
=&
\int_{0}^{\infty}t\int_{\R}\big|\partial_{t}T_{t}^{\a}f(x)\big|^2\,d\mu_{\a}(x)\,dt\\
=&
\sum_{n=0}^{\infty}\bigg(\int_{0}^{\infty}t\,e^{-2t\ll_{n}^{\a}}\,dt\bigg)(\ll_{n}^{\a})^2
\sum_{|k|=n}|\langle f,\ell_{k}^{\a}\rangle_{d\mu_{\a}}|^2\\
=&
\frac{1}{4}\|f\|_{L^2(d\mu_{\a})}^2.
\end{align*}
The conclusion follows.
\end {proof}

\begin{pro}\label{ogrh}
Assume that $\alpha\in (-1,\infty)^{d}$. The horizontal square functions satisfy
\begin{displaymath}
\begin {array}{lll}
\Big\|\big|\big(g_{H,T}^{1}(f),\ldots,g_{H,T}^{d}(f)\big)\big|_{\ell^2}\Big\|_{L^2(d\mu_{\a})}
&\simeq
\|f\|_{L^2(d\mu_{\a})},&\qquad
f\in\{\ell_{(0,\ldots,0)}^{\a}\}^{\perp}\subset L^2(d\mu_{\a}),\\
\Big\|\big|\big(g_{H,P}^{1}(f),\ldots,g_{H,P}^{d}(f)\big)\big|_{\ell^2}\Big\|_{L^2(d\mu_{\a})}
&\simeq
\|f\|_{L^2(d\mu_{\a})},&\qquad
f\in\{\ell_{(0,\ldots,0)}^{\a}\}^{\perp}\subset L^2(d\mu_{\a}),\\
\end {array}
\end{displaymath}
and for each $j=1,\ldots,d$,
\begin{displaymath}
\begin {array}{lll}
\Big\|\big|\big(g_{H,\widetilde T}^{j,1}(f),\ldots,g_{H,\widetilde T}^{j,d}(f)\big)\big|_{\ell^2}\Big\|_{L^2(d\mu_{\a})}
&\simeq
\|f\|_{L^2(d\mu_{\a})},&\qquad
f\in L^2(d\mu_{\a}),\\
\Big\|\big|\big(g_{H,\widetilde P}^{j,1}(f),\ldots,g_{H,\widetilde P}^{j,d}(f)\big)\big|_{\ell^2}\Big\|_{L^2(d\mu_{\a})}
&\simeq
\|f\|_{L^2(d\mu_{\a})},&\qquad
f\in L^2(d\mu_{\a}).
\end {array}
\end{displaymath}
\end{pro}

\begin {proof}
We give a detailed justification only for the last of the four stated relations. The remaining relations are proved in a similar manner.

Taking into account \eqref{pp}, \eqref{lka} and the identities (cf. \cite[(4.4)]{NS1})
\begin{align}\label{opdel}
\delta_{j}\ell_{k}^{\a}=-2\sqrt{k_{j}}x_{j}\ell_{k-e_{j}}^{\a+e_{j}},\qquad 
\delta_{j}^{*}(x_{j}\ell_{k-e_{j}}^{\a+e_{j}})=-2\sqrt{k_{j}}\ell_{k}^{\a},
\end{align}
we get
\begin{align*}
\delta_{i}\widetilde P_{t}^{\a,j}f(x)
=&
-2\sum_{n=0}^{\infty}e^{-t\sqrt{\ll_{n}^{\a}}}\sum_{|k|=n}\sqrt{k_{i}}\langle f,x_{j}\ell_{k-e_{j}}^{\a+e_{j}}\rangle_{d\mu_{\a}} x_{i}x_{j}\ell_{k-e_{j}-e_{i}}^{\a+e_{j}+e_{i}}(x),\qquad i\ne j,\\
\delta_{j}^{*}\widetilde P_{t}^{\a,j}f(x)
=&
-2\sum_{n=0}^{\infty}e^{-t\sqrt{\ll_{n}^{\a}}}\sum_{|k|=n}\sqrt{k_{j}}\langle f,x_{j}\ell_{k-e_{j}}^{\a+e_{j}}\rangle_{d\mu_{\a}}\ell_{k}^{\a}(x).
\end{align*}
Using these relations, the Fubini-Tonelli theorem, the fact that each of the systems $\{x_{i}x_{j}\ell_{k-e_{j}-e_{i}}^{\a+e_{j}+e_{i}} : k \in \N^d\}$, $i\ne j$ (with null functions excluded) is an orthonormal basis in $L^2(d\mu_{\a})$ and Parseval's identity we obtain for $i\ne j$
\begin{align*}
\|g_{H,\widetilde P}^{j,i}(f)\|^{2}_{L^2(d\mu_{\a})}
=&
\int_{0}^{\infty}t\int_{\R}\big|\delta_{i}\widetilde P_{t}^{\a,j}f(x)\big|^2\,d\mu_{\a}(x)\,dt\\
=&
4\sum_{n=0}^{\infty}\bigg(\int_{0}^{\infty}t\,e^{-2t\sqrt{\ll_{n}^{\a}}}\,dt\bigg)
\sum_{|k|=n}k_{i}|\langle f,x_{j}\ell_{k-e_{j}}^{\a+e_{j}}\rangle_{d\mu_{\a}}|^2\\
=&
\sum_{n=0}^{\infty}\frac{1}{\ll_{n}^{\a}}\sum_{|k|=n}k_{i}|\langle f,x_{j}\ell_{k-e_{j}}^{\a+e_{j}}\rangle_{d\mu_{\a}}|^2.
\end{align*}
Parallel arguments lead to 
\begin{align*}
\|g_{H,\widetilde P}^{j,j}(f)\|^{2}_{L^2(d\mu_{\a})}
=&
\sum_{n=0}^{\infty}\frac{1}{\ll_{n}^{\a}}\sum_{|k|=n}k_{j}|\langle f,x_{j}\ell_{k-e_{j}}^{\a+e_{j}}\rangle_{d\mu_{\a}}|^2.
\end{align*}
Combining these results we see that
\begin{align*}
\Big\|\big|\big(g_{H,\widetilde P}^{j,1}(f),\ldots,g_{H,\widetilde P}^{j,d}(f)\big)\big|_{\ell_{2}}\Big\|_{L^2(d\mu_{\a})}
=&
\sum_{n=0}^{\infty}\sum_{|k|=n}\frac{|k|}{\ll_{n}^{\a}}|\langle f,x_{j}\ell_{k-e_{j}}^{\a+e_{j}}\rangle_{d\mu_{\a}}|^2\\
=&
\sum_{n=1}^{\infty}\frac{n}{\ll_{n}^{\a}}\sum_{|k|=n}|\langle f,x_{j}\ell_{k-e_{j}}^{\a+e_{j}}\rangle_{d\mu_{\a}}|^2
\simeq
\|f\|^{2}_{L^2(d\mu_{\a})};
\end{align*}
the last relation is due to the fact that $\{x_{j}\ell_{k}^{\a+e_{j}} : k\in\N^{d}\}$ is an orthonormal basis in $L^2(d\mu_{\a})$. The proof is finished. 
\end {proof}

Next we prove Proposition \ref{prestowP}, that is the $g$-functions under consideration, viewed as vector-valued linear operators, are indeed associated with the relevant kernels. We adapt essentially the reasoning from \cite[Section 2]{StTo2} applied in the setting of Hermite function expansions, taking opportunity to introduce some simplifications. Treatment of each of the $g$-functions relies on similar arguments, therefore we give detailed proofs only in the two representative cases of $g_{V,P}$ and $g_{H,\widetilde T}^{j,j}$, leaving the remaining cases to the reader. When dealing with kernels with non-integrable singularities, applying Fubini's theorem or exchanging integration with differentiation is a delicate matter. Therefore below and also in Section \ref{ker} we provide fairly detailed explanations in the relevant places.

\begin {proof}[Proof of Proposition \ref{prestowP}; the case of $g_{V,P}$.]
By density arguments it suffices to show that
\begin{align}\label{s1}
\Big\langle\big\{\partial_{t}P_{t}^{\a} f\big\}_{t>0},h\Big\rangle_{L^2(\R,L^2(tdt))}
=
\bigg\langle\int_{\R}\big\{\partial_{t}P_{t}^{\a}(x,y)\big\}_{t>0}f(y)\,d\mu_{\a}(y),h\bigg\rangle_{L^2(\R,L^2(tdt))}
\end{align}
for every $f\in C^{\infty}_c(\R)$ and $h(x,t)=h_{1}(x)h_{2}(t)$, where $h_{1}\in C^{\infty}_c(\R)$,  $h_{2}\in C^{\infty}_c(\mathbb{R}_{+})$ and $\supp f\cap \supp h_{1}=\emptyset$ (the linear span of functions $h$ of this form is dense in $L^2\big((\supp f)^{C},d\mu_{\a}\otimes tdt\big)$). We start by considering the left-hand side of \eqref{s1},
\begin{align*}
\Big\langle\big\{&\partial_{t}P_{t}^{\a} f(x)\big\}_{t>0},h\Big\rangle_{L^2(\R,L^2(tdt))}\\
=&
\int_{0}^{\infty}t\overline{h_{2}(t)}\int_{\R}
\partial_{t}P_{t}^{\a}f(x)\overline{h_{1}(x)}\,d\mu_{\a}(x)\,dt\\
=&
\int_{0}^{\infty}t\overline{h_{2}(t)}\int_{\R}\bigg(-\sum_{n=0}^{\infty}\sqrt{\ll_{n}^{\a}}\,e^{-t\sqrt{\ll_{n}^{\a}}}\sum_{|k|=n}
\langle f,\ell_{k}^{\a}\rangle_{d\mu_{\a}}\ell_{k}^{\a}(x)\bigg)\overline{h_{1}(x)}\,d\mu_{\a}(x)\,dt\\
=&
-\int_{0}^{\infty}t\overline{h_{2}(t)}\sum_{n=0}^{\infty}\sqrt{\ll_{n}^{\a}}\,e^{-t\sqrt{\ll_{n}^{\a}}}\sum_{|k|=n}
\langle f,\ell_{k}^{\a}\rangle_{d\mu_{\a}}\overline{\langle h_{1},\ell_{k}^{\a}\rangle}_{d\mu_{\a}}\,dt.
\end{align*}
The change of order of integration in the first identity is justified by the Fubini theorem since
\begin{align*}
\int_{\R}\int_{0}^{\infty}t\big|\partial_{t}P_{t}^{\a}f(x)\big||h_{1}(x)h_{2}(t)|\,dt\,d\mu_{\a}(x)
\leq&
\,\big\|\partial_{t}P_{t}^{\a}f\big\|_{L^2(d\mu_{\a}\otimes tdt)}\|h_{1}\|_{L^2(d\mu_{\a})}\|h_{2}\|_{L^2(tdt)},
\end{align*} 
and the right-hand side here is finite because $f\mapsto \partial_{t}P_{t}^{\a}f$ is bounded from $L^2(d\mu_{\a})$ into $L^2(\R,L^2(tdt))$, see Proposition \ref{ogrv}. The second equality is obtained by exchanging the order of $\partial_{t}$ and $\sum$, and this is valid in view of \eqref{lka}.
The third identity is also a consequence of the Fubini theorem, and its application is legitimate since for $t>0$ 
\begin{align*} 
\int_{\R}&\bigg(\sum_{n=0}^{\infty}\sqrt{\ll_{n}^{\a}}\,e^{-t\sqrt{\ll_{n}^{\a}}}\sum_{|k|=n}
|\langle f,\ell_{k}^{\a}\rangle_{d\mu_{\a}}||\ell_{k}^{\a}(x)||h_{1}(x)|\bigg)\,d\mu_{\a}(x)\\
\lesssim&
\|f\|_{L^2(d\mu_{\a})}\sum_{n=0}^{\infty}\,e^{-t\sqrt{n}}\sum_{|k|=n}\int_{\R}|\ell_{k}^{\a}(x)h_{1}(x)|\,d\mu_{\a}(x)\\
\leq&
\|f\|_{L^2(d\mu_{\a})}\|h_{1}\|_{L^2(d\mu_{\a})}\sum_{n=0}^{\infty}e^{-t\sqrt{n}}(n+1)^{d}<\infty.
\end{align*} 

Now we focus on the right-hand side of \eqref{s1}. Interchanging the order of integrals we get
\begin{align*}
\bigg\langle&\int_{\R}\big\{\partial_{t}P_{t}^{\a}(x,y)\big\}_{t>0}f(y)\,d\mu_{\a}(y),h\bigg\rangle_{L^2(\R,L^2(tdt))}\\
=&
\int_{0}^{\infty}t\overline{h_{2}(t)}\int_{\R}\int_{\R}\partial_{t}P_{t}^{\a}(x,y)f(y)
\overline{h_{1}(x)}\,d\mu_{\a}(y)\,d\mu_{\a}(x)\,dt.
\end{align*}
Here application of the Fubini theorem is possible since
\begin{align*}
\int_{\R}&\int_{\R}\int_{0}^{\infty}t\big|\partial_{t}P_{t}^{\a}(x,y)f(y)\big||h_{1}(x)h_{2}(t)|
\,dt\,d\mu_{\a}(y)\,d\mu_{\a}(x)\\
\leq&
\|f\|_{\infty}\|h_{1}\|_{\infty}\|h_{2}\|_{L^2(tdt)}
\int_{\supp h_{1}}\int_{\supp f}\big\|\partial_{t}P_{t}^{\a}(x,y)\big\|_{L^2(tdt)}
\,d\mu_{\a}(y)\,d\mu_{\a}(x)\\
\lesssim&
\int_{\supp h_{1}}\int_{\supp f}
\frac{1}{\mu_{\alpha}(B(x,|y-x|))}\,d\mu_{\a}(y)\,d\mu_{\a}(x)<\infty,
\end{align*}
where we made use of the growth estimate for the kernel $\big\{\partial_{t}P_{t}^{\a}(x,y)\big\}$ proved in Section \ref{ker} below and the fact that the supports of $f$ and $h_{1}$ are disjoint and bounded.
Next, by the definition of $P_{t}^{\a}(x,y)$ and again Fubini's theorem we obtain
\begin{align*}
\int_{\R}&\int_{\R}\partial_{t}P_{t}^{\a}(x,y)f(y)
\overline{h_{1}(x)}\,d\mu_{\a}(y)\,d\mu_{\a}(x)\\
=&
\int_{\R}\int_{\R}\bigg(-\sum_{n=0}^{\infty}\sqrt{\ll_{n}^{\a}}\,e^{-t\sqrt{\ll_{n}^{\a}}}
\sum_{|k|=n}\ell_{k}^{\a}(x)\ell_{k}^{\a}(y)\bigg)f(y)\overline{h_{1}(x)}\,d\mu_{\a}(y)\,d\mu_{\a}(x)\\
=&
-\sum_{n=0}^{\infty}\sqrt{\ll_{n}^{\a}}\,e^{-t\sqrt{\ll_{n}^{\a}}}\sum_{|k|=n}
\langle f,\ell_{k}^{\a}\rangle_{d\mu_{\a}}\overline{\langle h_{1},\ell_{k}^{\a}\rangle}_{d\mu_{\a}}.
\end{align*}
Here the first identity is justified with the aid of \eqref{lka}. Application of Fubini's theorem in the second identity is legitimate since, with $t>0$ fixed,
\begin{align*}
\int_{\R}&\int_{\R}\bigg(\sum_{n=0}^{\infty}\sqrt{\ll_{n}^{\a}}\,e^{-t\sqrt{\ll_{n}^{\a}}}
\sum_{|k|=n}|\ell_{k}^{\a}(x)||\ell_{k}^{\a}(y)|\bigg)|f(y)h_{1}(x)|\,d\mu_{\a}(y)\,d\mu_{\a}(x)\\
\lesssim&
\|f\|_{L^2(d\mu_{\a})}\|h_{1}\|_{L^2(d\mu_{\a})}\sum_{n=0}^{\infty}e^{-t\sqrt{n}}(n+1)^{d}<\infty.
\end{align*}
Integrating the last identities against $\overline{h_{2}(t)}t\,dt$ we finally see that both sides of \eqref{s1} coincide.
\end {proof}

\begin {proof}[Proof of Proposition \ref{prestowP}; the case of $g_{H,\widetilde T}^{j,j}$.]
Density arguments reduce our task to showing that
\begin{align}\label{s2}
\Big\langle\big\{\delta_{j}^{*}\widetilde T_{t}^{\a,j} f\big\}_{t>0},h\Big\rangle_{L^2(\R,L^2(dt))}
=
\bigg\langle\int_{\R}\big\{\delta_{j}^{*}\widetilde G_{t}^{\a,j}(x,y)\big\}_{t>0}f(y)\,d\mu_{\a}(y),h\bigg\rangle_{L^2(\R,L^2(dt))}
\end{align}
for every $f\in C^{\infty}_c(\R)$ and $h(x,t)=h_{1}(x)h_{2}(t)$, where $h_{1}\in C^{\infty}_c(\R)$,  $h_{2}\in C^{\infty}_c(\mathbb{R}_{+})$ and $\supp f\cap \supp h_{1}=\emptyset$. We first deal with the left-hand side of \eqref{s2},
\begin{align*}
\Big\langle\big\{&\delta_{j}^{*}\widetilde T_{t}^{\a,j} f\big\}_{t>0},h\Big\rangle_{L^2(\R,L^2(dt))}\\
=&
\int_{0}^{\infty}\overline{h_{2}(t)}\int_{\R}\delta_{j}^{*}\widetilde T_{t}^{\a,j}f(x)\overline{h_{1}(x)}\,d\mu_{\a}(x)\,dt\\
=&
\int_{0}^{\infty}\overline{h_{2}(t)}\int_{\R}\bigg(-2\sum_{n=0}^{\infty}e^{-t\ll_{n}^{\a}}\sum_{|k|=n}
\sqrt{k_{j}}\langle f,x_{j}\ell_{k-e_{j}}^{\a+e_{j}}\rangle_{d\mu_{\a}}\ell_{k}^{\a}(x)\bigg)\overline{h_{1}(x)}
\,d\mu_{\a}(x)\,dt\\
=&
-2\int_{0}^{\infty}\overline{h_{2}(t)}\sum_{n=0}^{\infty}e^{-t\ll_{n}^{\a}}\sum_{|k|=n}\sqrt{k_{j}}\langle f,x_{j}\ell_{k-e_{j}}^{\a+e_{j}}\rangle_{d\mu_{\a}}
\overline{\langle h_{1},\ell_{k}^{\a}\rangle}_{d\mu_{\a}}\,dt.
\end{align*}
Changing the order of integrals in the first identity is justified by Fubini's theorem since
\begin{align*}
\int_{\R}\int_{0}^{\infty}\big|\delta_{j}^{*}\widetilde T_{t}^{\a,j}f(x)\big||h_{1}(x)h_{2}(t)|\,dt\,d\mu_{\a}(x)
\leq&
\big\|\delta_{j}^{*}\widetilde T_{t}^{\a,j}f\big\|_{L^2(d\mu_{\a}\otimes dt)}
\|h_{1}\|_{L^2(d\mu_{\a})}\|h_{2}\|_{L^2(dt)},
\end{align*}
and the right-hand side here is finite because $f\mapsto \delta_{j}^{*}\widetilde T_{t}^{\a,j}f$ is bounded from $L^2(d\mu_{\a})$ into $L^2(\R,L^2(dt))$, see Proposition \ref{ogrh}. The second equality is obtained with the aid of \eqref{opdel}, by differentiating the series \eqref{tfal} term by term (this is legitimate in view of \eqref{lka}).
The third identity is also a consequence of Fubini's theorem, and its application is valid since for a fixed $t>0$ 
\begin{align*}
\int_{\R}&\bigg(\sum_{n=0}^{\infty}e^{-t\ll_{n}^{\a}}\sum_{|k|=n}\sqrt{k_{j}}
\big|\langle f,x_{j}l_{k-e_{j}}^{\a+e_{j}}\rangle_{d\mu_{\a}}\big||\ell_{k}^{\a}(x)||h_{1}(x)|\bigg)\,d\mu_{\a}(x)\\
\lesssim&
\|f\|_{L^2(d\mu_{\a})}\|h_{1}\|_{L^2(d\mu_{\a})}
\sum_{n=0}^{\infty}e^{-tn}(n+1)^{d}<\infty.
\end{align*}

Now we consider the right-hand side of \eqref{s2}. Using the Fubini theorem we see that
\begin{align*}
\bigg\langle &\int_{\R}\big\{\delta_{j}^{*}\widetilde G_{t}^{\a,j}(x,y)\big\}_{t>0}f(y)\,d\mu_{\a}(y),h\bigg\rangle_{L^2(\R,L^2(dt))}\\
=&
\int_{0}^{\infty}\overline{h_{2}(t)}\int_{\R}\int_{\R}\delta_{j}^{*}\widetilde G_{t}^{\a,j}(x,y)f(y)
\overline{h_{1}(x)}\,d\mu_{\a}(y)\,d\mu_{\a}(x)\,dt.
\end{align*}
Application of Fubini's theorem is legitimate since
\begin{align*}
\int_{\R}&\int_{\R}\int_{0}^{\infty}\big|\delta_{j}^{*}\widetilde G_{t}^{\a,j}(x,y)f(y)\big||h_{1}(x)h_{2}(t)|
\,dt\,d\mu_{\a}(y)\,d\mu_{\a}(x)\\
\leq&
\|f\|_{\infty}\|h_{1}\|_{\infty}\|h_{2}\|_{L^2(dt)}
\int_{\supp h_{1}}\int_{\supp f}\big\|\delta_{j}^{*}\widetilde G_{t}^{\a,j}(x,y)\big\|_{L^2(dt)}\,d\mu_{\a}(y)\,d\mu_{\a}(x)\\
\lesssim&
\int_{\supp h_{1}}\int_{\supp f}
\frac{1}{\mu_{\alpha}(B(x,|y-x|))}\,d\mu_{\a}(y)\,d\mu_{\a}(x)<\infty;
\end{align*}
here we made use of the growth estimate for the kernel $\big\{\delta_{j}^{*}\widetilde G_{t}^{\a,j}(x,y)\big\}$ (see Theorem \ref{preg}) and the fact that $\supp f$ and $\supp h_{1}$ are disjoint and bounded.
Next, using \eqref{gfal}, \eqref{opdel} and Fubini's theorem we get
\begin{align*}
\int_{\R}&\int_{\R}\delta_{j}^{*}\widetilde G_{t}^{\a,j}(x,y)f(y)\overline{h_{1}(x)}\,d\mu_{\a}(y)\,d\mu_{\a}(x)\\
=&
\int_{\R}\int_{\R}\bigg(-2\sum_{n=0}^{\infty}e^{-t\ll_{n}^{\a}}\sum_{|k|=n}
\sqrt{k_{j}}\ell_{k}^{\a}(x)y_{j}\ell_{k-e_{j}}^{\a+e_{j}}(y)\bigg)f(y)
\overline{h_{1}(x)}\,d\mu_{\a}(y)\,d\mu_{\a}(x)\\
=&
-2\sum_{n=0}^{\infty}e^{-t\ll_{n}^{\a}}\sum_{|k|=n}\sqrt{k_{j}}
\langle f,x_{j}\ell_{k-e_{j}}^{\a+e_{j}}\rangle_{d\mu_{\a}}\overline{\langle h_{1},\ell_{k}^{\a}\rangle}_{d\mu_{\a}}.
\end{align*}
Here the first identity follows with the aid of \eqref{lka}.
Application of the Fubini theorem in the second identity is justified because
\begin{align*}
\int_{\R}&\int_{\R}\bigg(\sum_{n=0}^{\infty}e^{-t\ll_{n}^{\a}}\sum_{|k|=n}\sqrt{k_{j}}
|\ell_{k}^{\a}(x)|\big|y_{j}\ell_{k-e_{j}}^{\a+e_{j}}(y)\big|\bigg)|f(y)h_{1}(x)|\,d\mu_{\a}(y)\,d\mu_{\a}(x)\\
\lesssim&
\|f\|_{L^2(d\mu_{\a})}\|h_{1}\|_{L^2(d\mu_{\a})}\sum_{n=0}^{\infty}e^{-tn}(n+1)^{d}<\infty.
\end{align*}
Integrating the last identities against $\overline{h_{2}(t)}\,dt$ we get \eqref{s2}, as desired.
\end {proof}

\section{Kernel estimates}\label{ker}
\setcounter{equation}{0}
This section delivers proofs of the relevant kernel estimates for all considered square functions. Our methods are similar to those applied in \cite{NS1} and are based on Schl\"afli's integral representation for the modified Bessel function $I_{\nu}$ involved in the Laguerre heat kernel. Here, however, the technicalities are more sophisticated than in \cite{NS1} since now we are dealing with vector-valued kernels. From now on we always assume that $\a\in[-1\slash 2,\infty)^{d}$.

The Laguerre heat kernel $G_{t}^{\alpha}(x,y)$ may be expressed as, see  \cite[Section 5]{NS1},
\begin{equation}\label{gie}
G_{t}^{\alpha}(x,y)=\Big(\frac{1-\zeta^{2}}{2\zeta}\Big)^{d+|\alpha|}\int_{[-1,1]^{d}}\exp\Big(-{\frac{1}{4\zeta}}q_{+}(x,y,s)-\frac{\zeta}{4}q_{-}(x,y,s)\Big)\Pi_{\alpha}(ds),
\end{equation}
where 
\begin{align*}
q_{\pm}(x,y,s)=
|x|^2+|y|^2\pm 2\sum_{i=1}^{d}x_{i}y_{i}s_{i}
\end{align*}
and $\zeta=\tanh t$, $t>0$; equivalently
\begin{equation}\label{zz}
t=t(\zeta)=\frac{1}{2}\log\frac{1+\zeta}{1-\zeta},\qquad \zeta \in (0,1).
\end{equation}
The measure $\Pi_{\a}$ is a product of one-dimensional measures, $\Pi_{\a}=\bigotimes_{i=1}^{d}\Pi_{\a_{i}}$, where $\Pi_{\a_{i}}$ is given by the density
$$
\Pi_{\a_{i}}(ds_{i}) = \frac{(1-s_{i}^2)^{\a_{i}-1\slash 2}ds_{i}}{\sqrt{\pi} 2^{\a_{i}}\Gamma{(\a_{i}+1\slash 2)}},
    \qquad \a_{i} > -1\slash 2,
$$
and in the limiting case of $\a_{i}=-1\slash 2$, $\Pi_{-1\slash 2}=\big( \eta_{-1} + \eta_1 \big)\slash \sqrt{2\pi}$, with $\eta_{-1}$ and $\eta_{1}$ denoting point masses at $-1$ and $1$, respectively.

Recall that the kernels of the modified Laguerre semigroups are given by
\begin{equation}\label{falka}
\widetilde{G}^{\alpha,j}_t(x,y)= e^{-2t}x_jy_jG^{\alpha+e_j}_t(x,y),\qquad j=1,\ldots,d.
\end{equation}

To estimate expressions related to $G_{t}^{\alpha}(x,y)$ or $\widetilde{G}^{\alpha,j}_t(x,y)$ we will use several technical lemmas, which are gathered below. Some of the results were obtained elsewhere, but we state them anyway for the sake of completeness and reader's convenience.

To begin with, notice that
\begin{equation}\label{cal}
\int_{0}^{1}\zeta^{-1}\log\frac{1+\zeta}{1-\zeta}\,d\zeta < \infty \qquad \textrm{and}\qquad
\int_{0}^{1}\log\frac{1+\zeta}{1-\zeta}\,d\zeta < \infty ;
\end{equation}
this can be easily seen from the asymptotics
$$
\log\frac{1+\zeta}{1-\zeta}\sim \zeta,\quad\zeta\to 0^{+}
\qquad \textrm{and}\qquad
\log\frac{1+\zeta}{1-\zeta}\sim -\log(1-\zeta),\quad\zeta\to 1^{-}.
$$

\begin{lem}\label{obs}
Let $x,y\in\R$, $s\in[-1,1]^{d}$. Then for any $j=1,\ldots,d,$ we have
$$
|x_{j}\pm y_{j}s_{j}|\leq \sqrt{q_{\pm}(x,y,s)}\qquad \textrm{and}\qquad
|y_{j}\pm x_{j}s_{j}|\leq \sqrt{q_{\pm}(x,y,s)}.
$$
\end{lem}

\begin{lem}\label{oq}
Given $b\geq 0$ and $c>0$, we have
$$
\big(q_{\pm}(x,y,s)\big)^{b}\exp\big(-cAq_{\pm}(x,y,s)\big)\lesssim A^{-b}\exp\Big(\frac{-cA}{2}q_{\pm}(x,y,s)\Big),\qquad A>0,
$$
uniformly in $q_{\pm}$.
\end{lem}

\begin {proof}
It suffices to observe that\, $\sup_{u\geq 0}{u^{b}}\exp(-cu\slash 2)<\infty. $
\end {proof}

\begin{lem}$($\cite[Lemma 1.1]{StTo}$)$ \label{lem5.4}
Given $a>1$, we have
\begin{equation*}
\int_0^1\zeta^{-a}\exp(-T\zeta^{-1})\,d\zeta\lesssim T^{-a+1},
\qquad T>0.
\end{equation*}
\end{lem}

\begin{lem}\label{uw}
Let $C>0$ be fixed. Then
$$
\int_{0}^{1}\zeta^{-3}\log\frac{1+\zeta}{1-\zeta}\,\exp\Big(-\frac{C}{\zeta}q_{+}(x,y,s)\Big)\,d\zeta
\lesssim \frac{1}{q_{+}(x,y,s)}
$$
uniformly in $q_{+}$.
\end{lem}

\begin {proof}
Consider first $\z\in(1\slash 2,1)$. Bounding the factor $\zeta^{-2}$ by a constant and then using Lemma \ref{oq} (with $b=1$, $c=C$ and $A=\zeta^{-1}$) we see that
\begin{align*}
&\int_{1\slash 2}^{1}\zeta^{-3}\log\frac{1+\zeta}{1-\zeta}\,
\exp\Big(-\frac{C}{\zeta}q_{+}(x,y,s)\Big)\,d\zeta\\
&\lesssim
\frac{1}{q_{+}(x,y,s)}\int_{1\slash 2}^{1}\log\frac{1+\zeta}{1-\zeta}\,
\frac{q_{+}(x,y,s)}{\zeta}\exp\Big(-\frac{C}{\zeta}q_{+}(x,y,s)\Big)\,d\zeta\\
&\lesssim
\frac{1}{q_{+}(x,y,s)}\int_{1\slash 2}^{1}\log\frac{1+\zeta}{1-\zeta}\,d\zeta,\\
\end{align*}
where by \eqref{cal} the last integral is finite. On the other hand, observing that the function $\zeta\mapsto\frac{1}{\zeta}\log\frac{1+\zeta}{1-\zeta}$ is bounded on $(0,1\slash 2)$ and using Lemma \ref{lem5.4} (with $a=2$ and $T=Cq_{+}$) we get
$$
\int_{0}^{1\slash 2}\zeta^{-3}\log\frac{1+\zeta}{1-\zeta}\,
\exp\Big(-\frac{C}{\zeta}q_{+}(x,y,s)\Big)\,d\zeta
\lesssim
\int_{0}^{1\slash 2}
\zeta^{-2}\exp\Big(-\frac{C}{\zeta}q_{+}(x,y,s)\Big)\,d\zeta\\
\lesssim
\frac{1}{q_{+}(x,y,s)}.
$$
The conclusion follows.
\end {proof}

\begin{lem}\label{lemat}
If $x,x',y\in\R$ are such that $|x-y|>2|x-x'|$ and $\t=\ll x+(1-\ll)x'$ for some $\ll\in[0,1]$, then 
\begin{align*}
\frac{1}{4}q_{\pm}(x,y,s)
\leq
q_{\pm}(\t,y,s)
\leq
4q_{\pm}(x,y,s),\qquad s\in[-1,1]^d.
\end{align*}
The same holds after exchanging the roles of $x$ and $y$.
\end{lem}

\begin {proof}
Since $q_{\pm}(x,y,s)=q_{\mp}(x,y,-s)$, it is enough to consider $q_{+}$ only. We first show that $q_{+}(x,y,s)\leq 4q_{+}(\t,y,s)$. This will follow once we check that
\begin{align*}
\sqrt{q_{+}(x,y,s)}&\leq\sqrt{q_{+}(\t,y,s)}+|x-\t|,\\
|x-\t|&\leq\sqrt{q_{+}(\t,y,s)}.
\end{align*}
To verify the first inequality above we denote $ys=(y_{1}s_{1},\ldots,y_{d}s_{d})$ and write
\begin{align*}
\big(\sqrt{q_{+}(\t,y,s)}+|x-\t|\big)^2=&\,
q_{+}(\t,y,s)+|x-\t|^2+2|x-\t|\bigg(\sum_{j=1}^{d}(\t_{j}+y_{j}s_{j})^{2}+(1-s_{j}^2)y_{j}^2\bigg)^{1\slash 2}\\
\geq&
\,q_{+}(\t,y,s)+|x-\t|^2+2|x-\t||\t+ys|\\
=&
\,q_{+}(x,y,s)+2|x-\t||\t+ys|-2\sum_{j=1}^{d}(x_{j}-\t_{j})(\t_{j}+y_{j}s_{j});
\end{align*}
in view of the Schwarz inequality, the last expression is not less than $q_{+}(x,y,s)$. Checking the second inequality is even easier. Using the relation $|x-y|>2|x-x'|$ we get
\begin{align*}
|x-\t|\leq |x-x'|<\frac{1}{2}|x-y|\leq \frac{1}{2}(|x-\t|+|\t-y|),
\end{align*}
so $|x-\t|\leq|y-\t|$. On the other hand,
\begin{align*}
|y-\t|^2=\sum_{j=1}^{d}\t_{j}^2+y_{j}^2-2\t_{j}y_{j}\leq
\sum_{j=1}^{d}\t_{j}^2+y_{j}^2+2\t_{j}y_{j}s_{j}=
q_{+}(\t,y,s).
\end{align*}
Altogether, this gives $|x-\t|\leq\sqrt{q_{+}(\t,y,s)}$.

Proving $q_{+}(\t,y,s)\leq 4q_{+}(x,y,s)$ relies on exchanging the roles of $x$ and $\t$ in the above reasoning. Finally, the last assertion of the lemma follows by the symmetry $q_{\pm}(x,y,s)=q_{\pm}(y,x,s)$.
\end {proof}

\begin{lem}$($\cite[Lemma 5.3]{NS3}, \cite[Lemma 4]{NS2}$)$ \label{lem4}
Assume that $\alpha \in [-1\slash 2, \infty)^d$ and let $\delta,\kappa \in [0,\infty)^d$ be fixed.
Then for $x,y\in\R$, $x \neq y$,
\begin{equation*}
(x+y)^{2\delta} \int_{[-1,1]^d}
	 \big(q_{+}(x,y,s)\big)^{-d - |\alpha| -|\delta|}\,\Pi_{\alpha+\delta+\kappa}(ds)
\lesssim \frac{1}{\mu_{\alpha}(B(x,|y-x|))}
\end{equation*}
and
\begin{equation*}
 (x+y)^{2\delta}\int_{[-1,1]^d} 
    \big(q_{+}(x,y,s)\big)^{-d - |\alpha| -|\delta| - 1\slash 2}\,\Pi_{\alpha+\delta+\kappa}(ds)
\lesssim \frac{1}{|x-y|} \;
 \frac{1}{\mu_{\alpha}(B(x,|y-x|))},
\end{equation*}
where $(x+y)^{2\delta}=(x_{1}+y_{1})^{2\delta_{1}}\ldots (x_{d}+y_{d})^{2\delta_{d}}$.
\end{lem}

\begin{lem}\label{wch}
Let $F\colon (0,\infty)\times[-1,1]^{d}\mapsto\mathbb{R}$ be a function such that $F(\cdot,s)$ is continuously differentiable for each fixed $s$. Further, assume that for each $v>0$ there exists $a\in(0,v)$ and a function $f_{a,v}\in L^{1}(\Pi_{\a}(ds))$ such that $|\partial_{z}F(z,s)|\leq f_{a,v}(s)$ for all $z\in[a,v]$ and $s\in[-1,1]^{d}$. Then 
$$
\partial_{z}\int_{[-1,1]^{d}}F(z,s)\,\Pi_{\a}(ds)=
\int_{[-1,1]^{d}}\partial_{z}F(z,s)\,\Pi_{\a}(ds),\qquad z>0.
$$
\end{lem}

\begin {proof}
Since $f_{a,v}$ is integrable against $\Pi_{\a}$ we have
$$
\int_{a}^{v}\int_{[-1,1]^{d}}|\partial_{z}F(z,s)|\,\Pi_{\a}(ds)\,dz
\leq
(v-a)\int_{[-1,1]^{d}}f_{a,v}(s)\,\Pi_{\a}(ds)<\infty,\qquad v>a.
$$
Thus we may apply the Fubini theorem to obtain
\begin{align*}
\int_{a}^{v}\int_{[-1,1]^{d}}\partial_{z}F(z,s)\,\Pi_{\a}(ds)\,dz
=\int_{[-1,1]^{d}}\big(F(v,s)-F(a,s)\big)\,\Pi_{\a}(ds).
\end{align*}
Differentiating this identity in v gives the desired conclusion.
\end {proof}

To write estimates of expressions involving $G_{t}^{\alpha}(x,y)$ and $\widetilde{G}^{\alpha,j}_t(x,y)$, it is convenient to introduce the following abbreviations:
$$
\lo(\zeta)=\log\frac{1+\zeta}{1-\zeta}, \qquad \e\b=\exp\Big(-{\frac{1}{4\zeta}}q_{+}(x,y,s)-\frac{\zeta}{4}q_{-}(x,y,s)\Big).
$$
Also, we will often neglect the set of integration $[-1,1]^{d}$ in integrals against $\Pi_{\alpha}$ and will frequently write shortly $q_{+}$ and $q_{-}$ omitting the arguments.
\subsection{Vertical $g$-function based on $\mathbf{\{T_{t}^{\alpha}\}}$}\label{k1}

\begin {proof}[Proof of Theorem \ref{preg}; the case of $g_{V,T}$.]
We first deal with the growth estimate. Differentiating \eqref{gie} we get (passing with $\partial_{t}$ under the integral sign will be justified in a moment)
$$
\partial_{t}G_{t}^{\alpha}(x,y)=-\Big(\frac{1-\zeta^{2}}{2\zeta}\Big)^{d+|\alpha|}h(x,y,\zeta),
$$
where the auxiliary function $h$ is given by
$$
h(x,y,\zeta)=(d+|\alpha|)\frac{1+\zeta^{2}}{\zeta}\int \e\b\,\Pi_{\alpha}(ds)
+\frac{1-\zeta^{2}}{\zeta}\int \e\b\Big[-\frac{1}{4\zeta}q_{+}+\frac{\zeta}{4}q_{-}\Big]\,\Pi_{\alpha}(ds).
$$
Notice that $h$ depends on $\alpha$, but to shorten the notation we will not indicate that explicitly (a similar convention will be used to other auxiliary functions appearing in the sequel).

Estimating $1\pm\zeta^{2}$ by a constant and then applying Lemma \ref{oq} twice (first with $b=1$, $c=1\slash 4$, $A=\zeta^{-1}$ and then with $b=1$, $c=1\slash 4$, $A=\zeta$) we get
\begin{align*}
|h(x,y,\zeta)|&
\lesssim\zeta^{-1}\int\e\b\,\Pi_{\alpha}(ds)
+\zeta^{-1}\int\e\b\Big[\frac{1}{4\zeta}q_{+}+\frac{\zeta}{4}q_{-}\Big]\,\Pi_{\alpha}(ds)\\
&\lesssim
\zeta^{-1}\int \big(\e\b\big)^{1\slash 2}\,\Pi_{\alpha}(ds).
\end{align*}
Now changing the variable according to \eqref{zz} and then using sequently the above estimate, the Minkowski integral inequality, Lemma \ref{oq} (applied with $b=2d+2|\alpha|-1$, $c=1\slash 4$, $A=\zeta^{-1}$) and Lemma \ref{uw} (with $C=1\slash 8$) we obtain
\begin{align*}
\big\|\partial_{t}G_{t}^{\alpha}(x,y)\big\|_{L^{2}(tdt)} &=\bigg(\int_{0}^{1}\frac{1}{2}\lo(\zeta)\frac{1}{1-\zeta^{2}}\Big(\frac{1-\zeta^{2}}{2\zeta}\Big)^{2d+2|\alpha|}
|h(x,y,\zeta)|^{2}\,d\zeta\bigg)^{1\slash 2}\\
&\lesssim
\bigg(\int_{0}^{1}\lo(\zeta)\Big(\frac{1}{\zeta}\Big)^{2d+2|\alpha|}\Big(\frac{1}{\zeta}\Big)^{2}
\bigg(\int \big(\e\b\big)^{1\slash 2}\,
\Pi_{\alpha}(ds)\bigg)^{2}\,d\zeta\bigg)^{1\slash 2}\\
&\leq
\int\bigg(\int_{0}^{1}\lo(\zeta)\Big(\frac{1}{\zeta}\Big)^{2d+2|\alpha|-1}\Big(\frac{1}{\zeta}\Big)^{3}\e\b\,d\zeta\bigg)^{1\slash 2}\,\Pi_{\alpha}(ds)\\
&\lesssim
\int (q_{+})^{-d-|\alpha|+1\slash 2}
\bigg(\int_{0}^{1}\lo(\zeta)\Big(\frac{1}{\zeta}\Big)^{3}\big(\e\b\big)^{1\slash 2}\,d\zeta\bigg)^{1\slash 2}\,\Pi_{\alpha}(ds)\\
&\lesssim
\int (q_{+})^{-d-|\alpha|}\,\Pi_{\alpha}(ds).
\end{align*}
The growth estimate follows by Lemma \ref{lem4} (taken with $\delta=\kappa=0$).
Exchanging above $\partial_{t}$ with the integral sign was legitimate by virtue of Lemma \ref{wch} applied with $F(t,s)=\partial_{t}\e(\z(t),q_{\pm})$ and $f_{a,v}(s)\sim(q_{+}(x,y,s))^{-1}$. Indeed, $f_{a,v}\in L^{1}(\Pi_{\a}(ds))$ since $(q_{+})^{-1}\leq|x-y|^{-2}$ and 
$$
|\partial_{t}\e(\z(t),q_{\pm})|\lesssim
\frac{1}{\z(t)}\Big(\e(\z(t),q_{\pm})\Big)^{1\slash 2}\lesssim
\frac{1}{q_{+}}.
$$

It remains to prove the smoothness estimates. By symmetry reasons, it suffices to show that
$$
\big\|\partial_{t}G_{t}^{\a}(x,y)-\partial_{t}G_{t}^{\a}(x',y)\big\|_{L^2(tdt)}
\lesssim
\frac{|x-x'|}{|x-y|} \;
\frac{1}{\mu_{\alpha}(B(x,|y-x|))},\qquad |x-y|>2|x-x'|.
$$
We consider the derivatives
$$
\partial_{x_{i}}\partial_{t}G_{t}^{\alpha}(x,y)=-\Big(\frac{1-\zeta^{2}}{2\zeta}\Big)^{d+|\alpha|}
\partial_{x_{i}}h(x,y,\zeta),\qquad i=1,\ldots,d.
$$
Passing with $\partial_{x_{i}}$ under the integral sign (this can be easily justified with the aid of Lemma \ref{wch}, see the comment above) we get
\begin{align*}
\partial_{x_{i}}h(x,y,\zeta)=\,&
(d+|\alpha|)\frac{1+\zeta^{2}}{\zeta}\int \e\b\Big[-\frac{1}{2\zeta}(x_{i}+y_{i}s_{i})-\frac{\zeta}{2}(x_{i}-y_{i}s_{i})\Big]\,\Pi_{\alpha}(ds)\\
& 
+\frac{1-\zeta^{2}}{\zeta}\int \e\b\Big[\frac{1}{2\zeta}(x_{i}+y_{i}s_{i})+\frac{\zeta}{2}(x_{i}-y_{i}s_{i})\Big]\Big[\frac{1}{4\zeta}q_{+}-\frac{\zeta}{4}q_{-}\Big]\,\Pi_{\alpha}(ds)\\
&
+\frac{1-\zeta^{2}}{\zeta}\int \e\b\Big[-\frac{1}{2\zeta}(x_{i}+y_{i}s_{i})+\frac{\zeta}{2}(x_{i}-y_{i}s_{i})\Big]\,\Pi_{\alpha}(ds).
\end{align*}
Using Lemma \ref{obs} and then applying repeatedly Lemma \ref{oq} (with $b=1\slash 2$ or $b=1$) we obtain
\begin{align*}
|\partial_{x_{i}}h(x,y,\zeta)|
\lesssim \,&
\zeta^{-1}\int \e\b\Big[\frac{1}{\zeta}\sqrt{q_{+}}+\zeta\sqrt{q_{-}}\Big]\,\Pi_{\alpha}(ds)\\
& +\zeta^{-1}\int\ \e\b\Big[\frac{1}{\zeta}\sqrt{q_{+}}+\zeta\sqrt{q_{-}}\Big]\Big[\frac{1}{\zeta}q_{+}+\zeta q_{-}\Big]\,\Pi_{\alpha}(ds)\\
& +\zeta^{-1}\int \e\b\Big[\frac{1}{\zeta}\sqrt{q_{+}}+\zeta\sqrt{q_{-}}\Big]\,\Pi_{\alpha}(ds)\\
\lesssim\,&
\zeta^{-1}\int \big(\e\b\big)^{1\slash 4}\Big[\frac{1}{\sqrt{\zeta}}+\sqrt{\zeta}\Big]\,\Pi_{\alpha}(ds)\\
\lesssim\,&
\zeta^{-3\slash 2}
\int \big(\e\b\big)^{1\slash 4}\,\Pi_{\alpha}(ds).
\end{align*}
By the mean value theorem and the above estimates we have
\begin{align*}
|\partial_{t}G_{t}^{\a}(x,y)-\partial_{t}G_{t}^{\a}(x',y)|
\leq&
\,|x-x'||\nabla_{\!x}\partial_{t}G_{t}^{\a}(\t,y)|\\
\lesssim&
\,|x-x'|\sqrt{1-\z^2}\,\z^{-d-|\a|-3\slash 2}\int\big[\e\c\big]^{1\slash 4}\,\Pi_{\alpha}(ds),
\end{align*}
where $\t$ is a convex combination of $x$ and $x'$. Then assuming $|x-y|> 2|x-x'|$ and using Lemma \ref{lemat} shows that
\begin{align*}
|\partial_{t}G_{t}^{\a}(x,y)-\partial_{t}G_{t}^{\a}(x',y)|
\lesssim&
\,|x-x'|\sqrt{1-\z^2}\,\z^{-d-|\a|-3\slash 2}\int\big[\e\bbb\big]^{1\slash
16}\,\Pi_{\alpha}(ds).
\end{align*}
Now changing the variable according to \eqref{zz} and then using the above estimate, the Minkowski inequality, Lemma \ref{oq} (specified to $b=2d+2|\alpha|$, $c=1\slash 32$, $A=\zeta^{-1}$) and Lemma \ref{uw} (with $C=1\slash 64$) we get
\begin{align*}
\big\|\partial_{t}&G_{t}^{\a}(x,y)-\partial_{t}G_{t}^{\a}(x',y)\big\|_{L^2(tdt)}\\
&\lesssim
|x-x'|\bigg(\int_{0}^{1}\lo(\zeta)\Big(\frac{1}{\zeta}\Big)^{2d+2|\alpha|+3}\bigg(\int \big(\e\b\big)^{1\slash 16}
\,\Pi_{\alpha}(ds)\bigg)^{2}\,d\zeta\bigg)^{1\slash 2}\\
&\leq
|x-x'|\int\bigg(\int_{0}^{1}\lo(\zeta)\Big(\frac{1}{\zeta}\Big)^{2d+2|\alpha|}\Big(\frac{1}{\zeta}\Big)^{3}\big(\e\b\big)^
{1\slash 8}\,d\zeta\bigg)^{1\slash 2}\,\Pi_{\alpha}(ds)\\
&\lesssim
|x-x'|\int (q_{+})^{-d-|\alpha|}
\bigg(\int_{0}^{1}\lo(\zeta)\Big(\frac{1}{\zeta}\Big)^{3}
\big(\e\b\big)^{1\slash 16}
\,d\zeta\bigg)^{1\slash 2}\,\Pi_{\alpha}(ds)\\
&\lesssim
|x-x'|\int (q_{+})^{-d-|\alpha|-1\slash 2}\,\Pi_{\alpha}(ds).
\end{align*}
This combined with Lemma \ref{lem4} (specified to $\delta=\kappa=0$) implies the desired estimate.
\newline
The proof of the case of $g_{V,T}$ in Theorem \ref{preg} is complete.
\end {proof}
\subsection{Horizontal $g$-functions based on  $\mathbf{\{T_{t}^{\alpha}\}}$}\label{k2}

\begin {proof}[Proof of Theorem \ref{preg}; the case of $g_{H,T}^{j}$.]
We start with proving the growth condition. Without any loss of generality we may focus only on the case $j=1$. With the aid of Lemma \ref{wch} we get
$$
\delta_{1,x}G_{t}^{\alpha}(x,y)=\Big(\frac{1-\zeta^{2}}{2\zeta}\Big)^{d+|\alpha|}h(x,y,\zeta),
$$
where the auxiliary function $h$ is now given by
$$
h(x,y,\zeta)=-\int\e\b\Big[\frac{1}{2\zeta}(x_{1}+y_{1}s_{1})+\frac{\zeta}{2}(x_{1}-y_{1}s_{1})\Big]\,\Pi_{\alpha}(ds)
+x_{1}\int\e\b\,\Pi_{\alpha}(ds).
$$

Using Lemma \ref{obs}, the fact that $x_{1}\leq\sqrt{q_{+}}+\sqrt{q_{-}}$ and then Lemma \ref{oq} (with $b=1\slash 2$, $A=\zeta^{-1}$ and $A=\zeta$, respectively) we get
\begin{align*}
|h(x,y,\zeta)|
&\leq
\int\e\b\Big[\frac{1}{\zeta}\sqrt{q_{+}}+\zeta\sqrt{q_{-}}\Big]\,\Pi_{\alpha}(ds)
+\int(\sqrt{q_{+}}+\sqrt{q_{-}})\e\b\,\Pi_{\alpha}(ds)\\
&\lesssim
\int\big(\e\b\big)^{1\slash 2}\Big[\frac{1}{\sqrt{\zeta}}+\sqrt{\zeta}\Big]\,\Pi_{\alpha}(ds)
+\int\big(\e\b\big)^{1\slash 2}\Big[\sqrt{\zeta}+\frac{1}{\sqrt{\zeta}}\Big]\,\Pi_{\alpha}(ds)\\
&\lesssim
\zeta^{-1\slash 2}\int\big(\e\b\big)^{1\slash 2}\,\Pi_{\alpha}(ds).
\end{align*}
Now changing the variable according to \eqref{zz} and then using sequently the above estimate, the Minkowski inequality, Lemma \ref{oq} (applied with $b=2d+2|\alpha|-1$, $c=1\slash 4$, $A=\zeta^{-1}$) and Lemma \ref{lem5.4} (with $a=2$ and $T=\frac{1}{8}q_{+}$) we obtain
\begin{align*}
\big\|\delta_{1,x}G_{t}^{\alpha}(x,y)\big\|_{L^{2}(dt)}=&
\bigg(\int_{0}^{1}\frac{1}{1-\zeta^{2}}\Big(\frac{1-\zeta^{2}}{2\zeta}\Big)^{2d+2|\alpha|}|h(x,y,\zeta)|^{2}\,d\zeta\bigg)
^{1\slash 2}\\
\lesssim&
\bigg(\int_{0}^{1}\Big(\frac{1}{\zeta}\Big)^{2d+2|\alpha|}\frac{1}{\zeta}
\bigg(\int\Big(\e\b\Big)^{1\slash 2}\,\Pi_{\alpha}(ds)\bigg)^{2}\,d\zeta\bigg)^{1\slash 2}\\
\leq&
\int\bigg(\int_{0}^{1}\Big(\frac{1}{\zeta}\Big)^{2d+2|\alpha|-1}\Big(\frac{1}{\zeta}\Big)^{2}\e\b\,d\zeta\bigg)
^{1\slash 2}\Pi_{\alpha}(ds)\\
\lesssim&
\int(q_{+})^{-d-|\alpha|+1\slash 2}\bigg(\int_{0}^{1}\Big(\frac{1}{\zeta}\Big)^{2}\big(\e\b\big)^{1\slash 2}\,d\zeta\bigg)^{1\slash 2}\,\Pi_{\alpha}(ds)\\
\lesssim&
\int(q_{+})^{-d-|\alpha|}\,\Pi_{\alpha}(ds).
\end{align*}
In view of Lemma \ref{lem4} (specified to $\delta=\kappa=0$) the growth estimate follows.

To prove the smoothness estimates we first show the bound
$$
\big\|\delta_{1,x}G_{t}^{\a}(x,y)-\delta_{1,x}G_{t}^{\a}(x',y)\big\|_{L^2(dt)}
\lesssim
\frac{|x-x'|}{|x-y|} \;
\frac{1}{\mu_{\alpha}(B(x,|y-x|))},\qquad |x-y|>2|x-x'|.
$$
To do that, we analyze the gradient
$$
\nabla_{\!x}\delta_{1,x}G_{t}^{\alpha}(x,y)=
\Big(\frac{1-\zeta^{2}}{2\zeta}\Big)^{d+|\alpha|}\nabla_{\!x}h(x,y,\zeta).
$$ 
While treating $\partial_{x_{i}}h(x,y,\z)$ it is natural to distinguish two cases.
\newline
{\bf Case 1:} $\mathbf {i\ne 1.}$ (this case appears only for $d\geq 2$) By symmetry reasons, we may restrict to $i=2$. Then
\begin{align*}
\partial_{x_{2}}&h(x,y,\zeta)\\
=&\int\e\b\Big[\frac{1}{2\zeta}(x_{2}+y_{2}s_{2})+\frac{\zeta}{2}(x_{2}-y_{2}s_{2})\Big]
\Big[\frac{1}{2\zeta}(x_{1}+y_{1}s_{1})+\frac{\zeta}{2}(x_{1}-y_{1}s_{1})\Big]\,\Pi_{\alpha}(ds)\\
&-x_{1}\int\e\b\Big[\frac{1}{2\zeta}(x_{2}+y_{2}s_{2})+\frac{\zeta}{2}(x_{2}-y_{2}s_{2})\Big]\,\Pi_{\alpha}(ds).
\end{align*}
Using Lemma \ref{obs}, the fact that $x_{1}\leq\sqrt{q_{+}}+\sqrt{q_{-}}$ and then Lemma \ref{oq} (specified to $b=1\slash 2$, $A=\zeta^{-1}$ and $A=\zeta$, respectively) we obtain
\begin{align*}
|\partial_{x_{2}}h(x,y,\zeta)|\leq&
\int\e\b\Big[\frac{1}{\zeta}\sqrt{q_{+}}+\zeta \sqrt{q_{-}}\Big]^{2}\,\Pi_{\alpha}(ds)\\
&+\int(\sqrt{q_{+}}+\sqrt{q_{-}})\e\b\Big[\frac{1}{\zeta}\sqrt{q_{+}}+\zeta \sqrt{q_{-}}\Big]\,\Pi_{\alpha}(ds)\\
\lesssim&
\int\big(\e\b\big)^{1\slash 4}\Big[\frac{1}{\sqrt{\zeta}}+\sqrt{\z}\Big]^{2}\,\Pi_{\a}(ds)\\
\lesssim&
\,\zeta^{-1}\int\big(\e\b\big)^{1\slash 4}\,\Pi_{\a}(ds).
\end{align*}
{\bf Case 2:} $\mathbf{i=1.}$ An elementary computation shows that
\begin{align*}
\partial_{x_{1}}h(x,y,\zeta)=&
\int\e\b
\Big[\frac{1}{2\zeta}(x_{1}+y_{1}s_{1})+\frac{\zeta}{2}(x_{1}-y_{1}s_{1})\Big]^{2}\,\Pi_{\alpha}(ds)\\
&-\int\e\b\Big[\frac{1}{2\z}+\frac{\z}{2}\Big]\,\Pi_{\a}(ds)
+\int\e\b\,\Pi_{\a}(ds)\\
&-x_{1}\int\e\b
\Big[\frac{1}{2\zeta}(x_{1}+y_{1}s_{1})+\frac{\zeta}{2}(x_{1}-y_{1}s_{1})\Big]\,\Pi_{\alpha}(ds).
\end{align*}
Proceeding in the same way as before we get
\begin{align*}
|\partial_{x_{1}}h(x,y,\zeta)|\lesssim
\z^{-1}\int\big(\e\b\big)^{1\slash 4}\,\Pi_{\a}(ds).
\end{align*}
In view of the mean value theorem, the above estimates and Lemma \ref{lemat} we have
\begin{align*}
|\delta_{1,x}G_{t}^{\a}(x,y)-\delta_{1,x}G_{t}^{\a}(x',y)|
\leq&
\,|x-x'||\nabla_{\!x}\delta_{1,x}G_{t}^{\a}(\t,y)|\\
\lesssim&
\,|x-x'|\sqrt{1-\z^2}\,\z^{-d-|\a|-1}\int\big[\e\c\big]^{1\slash 4}\,\Pi_{\a}(ds)\\
\leq&
\,|x-x'|\sqrt{1-\z^2}\,\z^{-d-|\a|-1}\int\big[\e\bbb\big]^{1\slash 16}\,\Pi_{\a}(ds),
\end{align*}
provided that $|x-y|>2|x-x'|$.
Now changing the variable as in \eqref{zz} and then using sequently the above estimate, the Minkowski inequality, Lemma \ref{oq} (taken with $b=2d+2|\alpha|$, $c=1\slash 32$, $A=\zeta^{-1}$) and Lemma \ref{lem5.4} (with $a=2$ and $T=\frac{1}{64}q_{+}$) we obtain 
\begin{align*}
\big\|\delta_{1,x}&G_{t}^{\a}(x,y)-\delta_{1,x}G_{t}^{\a}(x',y)\big\|_{L^2(dt)}\\
\lesssim&
\,|x-x'|\bigg(\int_{0}^{1}\Big(\frac{1}{\z}\Big)^{2d+2|\a|+2}\bigg(
\int\big(\e\b\big)^{1\slash 16}\,\Pi_{\a}(ds)\bigg)^{2}\,d\z\bigg)^{1\slash 2}\\
\leq&
\,|x-x'|\int\bigg(\int_{0}^{1}\Big(\frac{1}{\z}\Big)^{2d+2|\a|}\Big(\frac{1}{\z}\Big)^{2}
\big(\e\b\big)^{1\slash 8}\,d\z\bigg)^{1\slash 2}\,\Pi_{\a}(ds)\\
\lesssim&
\,|x-x'|\int(q_{+})^{-d-|\a|}\bigg(\int_{0}^{1}\Big(\frac{1}{\z}\Big)^{2}\big(\e\b\big)^{1\slash 16}\,d\z\bigg)^{1\slash 2}\,\Pi_{\a}(ds)\\
\lesssim&
\,|x-x'|\int(q_{+})^{-d-|\a|-1\slash 2}\,\Pi_{\a}(ds).
\end{align*}
This together with Lemma \ref{lem4} (specified to $\delta=\kappa=0$) gives the desired bound. 

The proof will be finished once we show that
$$
\big\|\delta_{1,x}G_{t}^{\a}(x,y)-\delta_{1,x}G_{t}^{\a}(x,y')\big\|_{L^2(dt)}\\
\lesssim
\frac{|y-y'|}{|x-y|} \;
\frac{1}{\mu_{\alpha}(B(x,|y-x|))},\qquad |x-y|>2|y-y'|.
$$
Taking into account the above considerations, it suffices to verify that
$$
|\partial_{y_{i}}h(x,y,\zeta)|\lesssim
\z^{-1}\int\big(\e\b\big)^{1\slash 4}\,\Pi_{\a}(ds),\qquad i=1,\ldots,d.
$$
Again, it is convenient to consider two cases.
\newline
{\bf Case 1:} $\mathbf{i\ne 1.}$ With no loss of generality we may assume that $i=2$. Then
\begin{align*}
\partial_{y_{2}}&h(x,y,\zeta)\\
=&\int\e\b\Big[\frac{1}{2\z}(y_{2}+x_{2}s_{2})+\frac{\z}{2}(y_{2}-x_{2}s_{2})\Big]
\Big[\frac{1}{2\zeta}(x_{1}+y_{1}s_{1})+\frac{\zeta}{2}(x_{1}-y_{1}s_{1})\Big]\,\Pi_{\alpha}(ds)\\
&
-x_{1}\int\e\b\Big[\frac{1}{2\z}(y_{2}+x_{2}s_{2})+\frac{\z}{2}(y_{2}-x_{2}s_{2})\Big]\,\Pi_{\alpha}(ds).
\end{align*}
Using Lemma \ref{obs}, the fact that $x_{1}\leq\sqrt{q_{+}}+\sqrt{q_{-}}$ and then Lemma \ref{oq} (applied with $b=1\slash 2$, $A=\zeta^{-1}$ and $A=\zeta$, respectively) we get
\begin{align*}
|\partial_{y_{2}}h(x,y,\zeta)|\leq&
\int\e\b\Big[\frac{1}{\zeta}\sqrt{q_{+}}+\zeta \sqrt{q_{-}}\Big]^{2}\,\Pi_{\alpha}(ds)\\
&+\int(\sqrt{q_{+}}+\sqrt{q_{-}})\e\b\Big[\frac{1}{\zeta}\sqrt{q_{+}}+\zeta \sqrt{q_{-}}\Big]\,\Pi_{\alpha}(ds)\\
\lesssim&
\int\big(\e\b\big)^{1\slash 4}\Big[\frac{1}{\sqrt{\zeta}}+\sqrt{\z}\Big]^{2}\,\Pi_{\a}(ds)\\
\lesssim&
\,\zeta^{-1}\int\big(\e\b\big)^{1\slash 4}\,\Pi_{\a}(ds).
\end{align*}
{\bf Case 2:} $\mathbf{i=1.}$ An easy computation produces
\begin{align*}
\partial_{y_{1}}&h(x,y,\zeta)\\
=&\int\e\b\Big[\frac{1}{2\z}(y_{1}+x_{1}s_{1})+\frac{\z}{2}(y_{1}-x_{1}s_{1})\Big]
\Big[\frac{1}{2\zeta}(x_{1}+y_{1}s_{1})+\frac{\zeta}{2}(x_{1}-y_{1}s_{1})\Big]\,\Pi_{\alpha}(ds)\\
&-x_{1}\int\e\b\Big[\frac{1}{2\z}(y_{1}+x_{1}s_{1})+\frac{\z}{2}(y_{1}-x_{1}s_{1})\Big]\,\Pi_{\alpha}(ds)\\
&-\int\e\b\Big[\frac{1}{2\z}s_{1}-\frac{\z}{2}s_{1}\Big]\,\Pi_{\alpha}(ds).
\end{align*}
Parallel arguments to those from Case 1 and the fact that $|s_{1}|\leq 1$ lead to the bound 
$$
|\partial_{y_{1}}h(x,y,\zeta)|\lesssim
\zeta^{-1}\int\big(\e\b\big)^{1\slash 4}\,\Pi_{\a}(ds).
$$
The proof of the case of $g_{H,T}^{j}$ in Theorem \ref{preg} is complete.
\end {proof}
\subsection{$g$-functions based on  $\mathbf{\{P_{t}^{\alpha}\}}$}\label{k3}
$\newline$
In this section our aim is to prove the relevant estimates for the vector-valued kernels $\big\{\partial_{t}P_{t}^{\a}(x,y)\big\}_{t>0}$ and $\big\{\delta_{j,x}P_{t}^{\a}(x,y)\big\}_{t>0}$, $j=1,\ldots,d$.
This will be achieved by means of the subordination principle and the kernels' estimates already obtained in Sections \ref{k1} and \ref{k2}; see \cite[p.\,114]{StTo2} for a related discussion concerning vertical $g$-functions based on general contraction semigroups.

\begin {proof}[Proof of Theorem \ref{preg}; the case of $g_{V,P}$.]
We first deal with the growth estimate. By the subordination principle
$$
\partial_{t}P_{t}^{\a}(x,y)=
\frac{1}{\sqrt{\pi}}\int_{0}^{\infty}\frac{e^{-u}}{\sqrt{u}}\partial_{t}\big(G_{t^{2}\slash 4u}^{\a}(x,y)\big)\,du=
\frac{1}{\sqrt{\pi}}\int_{0}^{\infty}\frac{e^{-u}}{\sqrt{u}}
\frac{t}{2u}\partial_{\tau}G_{\tau}^{\a}(x,y)\big|_{\tau=t^{2}\slash 4u}\,du.
$$
Exchanging above $\partial_{t}$ with the integral sign is justified by Lemma \ref{wch} adjusted by replacing $[-1,1]^{d}$ with $(0,\infty)$ and $\Pi_{\a}(ds)$ with $du$ and then applied with $F(t,u)=\frac{e^{-u}}{\sqrt{u}}G_{t^2\slash 4u}^{\a}(x,y)$ and $f_{a,v}(u)\sim \frac{e^{-u}}{\sqrt{u}}|x-y|^{-2d-2|\a|-2}$. Indeed, $f_{a,v}\in L^{1}(du)$ and (see Section \ref{k1})
\begin{align*}
\frac{t}{2u}\Big|\partial_{\tau}G_{\tau}^{\a}(x,y)\big|_{\tau=t^2\slash 4u}\Big|
\lesssim&
\frac{t}{u}\Big(\frac{1-\z^{2}(t^2\slash 4u)}{2\z(t^2\slash 4u)}\Big)^{d+|\a|}\frac{1}{\z(t^2\slash 4u)}\int\big(\e(\z(t^2\slash 4u),q_{\pm})\big)^{1\slash 2}\,\Pi_{\a}(ds)\\
\lesssim&
\,t\frac{\sqrt{1-\z^{2}(t^2\slash 4u)}}{u}|x-y|^{-2d-2|\a|-2}\\
\lesssim&
\,|x-y|^{-2d-2|\a|-2};
\end{align*}
the last estimate holds because the function $(t,u)\mapsto t\frac{\sqrt{1-\z^{2}(t^2\slash 4u)}}{u}$ is bounded on $[a,v]\times(0,\infty)$.

Now using the Minkowski inequality and then changing the variable $s=t^{2}\slash 4u$ we obtain
\begin{align*}
\big\|\partial_{t}P_{t}^{\a}(x,y)\big\|_{L^{2}(tdt)}=&
\bigg(\int_{0}^{\infty}t\bigg(\frac{1}{\sqrt{\pi}}\int_{0}^{\infty}\frac{e^{-u}}{\sqrt{u}}
\frac{t}{2u}\partial_{\tau}G_{\tau}^{\a}(x,y)\big|_{\tau=t^{2}\slash 4u}\,du\bigg)^{2}\,dt\bigg)^{1\slash 2}\\
\leq&
\int_{0}^{\infty}\bigg(\int_{0}^{\infty}t\frac{e^{-2u}}{u}\bigg(\frac{t}{2u}
\partial_{\tau}G_{\tau}^{\a}(x,y)\big|_{\tau=t^{2}\slash 4u}\bigg)^{2}\,dt\bigg)^{1\slash 2}\,du\\
=&
\int_{0}^{\infty}\frac{e^{-u}}{\sqrt{u}}\bigg(\int_{0}^{\infty}2s\big(\partial_{s}G_{s}^{\a}(x,y)\big)^{2}\,ds\bigg)^{1\slash 2}\,du.
\end{align*}
Combining this with the growth estimate from Theorem \ref{preg} (the case of $g_{V,T}$) we get the growth estimate for $\big\{\partial_{t}P_{t}^{\a}(x,y)\big\}$.

It remains to prove the smoothness estimates. By symmetry reasons, it suffices to show that
$$
\big\|\partial_{t}P_{t}^{\a}(x,y)-\partial_{t}P_{t}^{\a}(x',y)\big\|_{L^2(tdt)}
\lesssim
\frac{|x-x'|}{|x-y|} \;
\frac{1}{\mu_{\alpha}(B(x,|y-x|))},\qquad |x-y|>2|x-x'|.
$$
Using the Minkowski inequality and then changing the variable $s=t^{2}\slash 4u$ we get
\begin{align*}
\big\|\partial_{t}&P_{t}^{\a}(x,y)-\partial_{t}P_{t}^{\a}(x',y)\big\|_{L^2(tdt)}\\
=&
\bigg(\int_{0}^{\infty}t\bigg(\frac{1}{\sqrt{\pi}}\int_{0}^{\infty}\frac{e^{-u}}{\sqrt{u}}
\frac{t}{2u}\big(\partial_{\tau}G_{\tau}^{\a}(x,y)
-\partial_{\tau}G_{\tau}^{\a}(x',y)
\big)\big|_{\tau=t^{2}\slash 4u}\,du\bigg)^{2}\,dt\bigg)^{1\slash 2}\\
\leq&
\int_{0}^{\infty}\bigg(\int_{0}^{\infty}t\frac{e^{-2u}}{u}\bigg(\frac{t}{2u}
\big(\partial_{\tau}G_{\tau}^{\a}(x,y)
-\partial_{\tau}G_{\tau}^{\a}(x',y)\big)\big|_{\tau=t^{2}\slash 4u}\bigg)^{2}\,dt\bigg)^{1\slash 2}\,du\\
=&
\int_{0}^{\infty}\frac{e^{-u}}{\sqrt{u}}\bigg(\int_{0}^{\infty}
2s\big(\partial_{s}G_{s}^{\a}(x,y)-\partial_{s}G_{s}^{\a}(x',y)\big)^{2}\,ds\bigg)^{1\slash 2}\,du.
\end{align*}
Now the smoothness estimate from Theorem \ref{preg} (the case of $g_{V,T}$) comes into play and the conclusion follows. This finishes proving the case of $g_{V,P}$ in Theorem \ref{preg}.
\end {proof}


\begin {proof}[Proof of Theorem \ref{preg}; the case of $g_{H,P}^{j}$.]
As usually, we first show the growth condition. By the subordination principle and with the aid of (suitably adjusted) Lemma \ref{wch} we get
$$
\delta_{j,x}P_{t}^{\a}(x,y)=
\frac{1}{\sqrt{\pi}}\int_{0}^{\infty}\frac{e^{-u}}{\sqrt{u}}\delta_{j,x}G_{t^{2}\slash 4u}^{\a}(x,y)\,du.
$$
The Minkowski inequality and the change of variable $s=t^{2}\slash 4u$ lead to
$$
\big\|\delta_{j,x}P_{t}^{\a}(x,y)\big\|_{L^{2}(tdt)}
\leq
\frac{1}{\sqrt{\pi}}\int_{0}^{\infty}\frac{e^{-u}}{\sqrt{u}}\bigg(\int_{0}^{\infty}2u\big(\delta_{j,x}G_{s}^{\a}(x,y)\big)^{2}\,ds\bigg)^{1\slash 2}\,du,
$$
which in view of Theorem \ref{preg} (the case of $g_{H,T}^{j}$) gives the desired conclusion. 

To prove the smoothness estimates it is enough to verify that 
$$
\big\|\delta_{j,x}P_{t}^{\a}(x,y)-\delta_{j,x}P_{t}^{\a}(x',y)\big\|_{L^2(tdt)}
\lesssim
\frac{|x-x'|}{|x-y|} \;
\frac{1}{\mu_{\alpha}(B(x,|y-x|))},\qquad |x-y|>2|x-x'|;
$$
parallel arguments show that the other smoothness estimate holds. Using the Minkowski inequality and then changing the variable as before we obtain
\begin{align*}
\big\|\delta_{j,x}&P_{t}^{\a}(x,y)-\delta_{j,x}P_{t}^{\a}(x',y)\big\|_{L^2(tdt)}\\
\leq&
\frac{1}{\sqrt{\pi}}\int_{0}^{\infty}\frac{e^{-u}}{\sqrt{u}}\bigg(\int_{0}^{\infty}
2u\big(\delta_{j,x}G_{s}^{\a}(x,y)
-\delta_{j,x}G_{s}^{\a}(x',y)\big)^{2}\,ds\bigg)^{1\slash 2}\,du.
\end{align*}
The relevant bound follows now from the already proved case of $g_{H,T}^{j}$ in Theorem \ref{preg}. 
\end {proof}
\subsection{Vertical $g$-functions based on  $\mathbf{\{\widetilde T_{t}^{\a,j}\}}$}\label{k4}

\begin {proof}[Proof of Theorem \ref{preg}; the case of $g_{V,\widetilde T}^{j}$.]
Without any loss of generality we may focus only on the case $j=1$. Differentiating \eqref{falka} we get
\begin{equation}\label{nier3}
\partial_{t}\widetilde G_{t}^{\a,1}(x,y)=
e^{-2t}x_{1}y_{1}\big(-2G_{t}^{\a+e_{1}}(x,y)+\partial_{t}G_{t}^{\a+e_{1}}(x,y)\big).
\end{equation}
Using the estimates of $G_{t}^{\a}(x,y)$ and $\partial_{t}G_{t}^{\a}(x,y)$ from Section \ref{k1} we see that
\begin{align*}
\big|\partial_{t}\widetilde G_{t}^{\a,1}(x,y)\big|
\lesssim&
\,x_{1}y_{1}\Big(\frac{1-\z^2}{2\z}\Big)^{d+|\a|+1}\int\e\b\,\Pi_{\a+e_{1}}(ds)\\
&+
\,x_{1}y_{1}\Big(\frac{1-\z^2}{2\z}\Big)^{d+|\a|+1}
\z^{-1}\int\big(\e\b\big)^{1\slash 2}\,\Pi_{\a+e_{1}}(ds)\\
\lesssim&
\,x_{1}y_{1}\Big(\frac{1-\z^2}{2\z}\Big)^{d+|\a|+1}\z^{-1}\int\big(\e\b\big)^{1\slash 2}\,\Pi_{\a+e_{1}}(ds).
\end{align*}
Changing the variable as in \eqref{zz} and then applying the above estimate, the Minkowski inequality, Lemma \ref{oq} (specified to $b=2d+2|\a|+1$, $c=1\slash 4$, $A=\z^{-1}$) and Lemma \ref{uw} (with $C=1\slash 8$) we obtain
\begin{align*}
\big\|\partial_{t}&\widetilde G_{t}^{\alpha,1}(x,y)\big\|_{L^{2}(tdt)}\\
=&
\bigg(\int_{0}^{1}\frac{1}{2}\lo(\z)\frac{1}{1-\z^2}\big(\partial_{t}\widetilde G_{t}^{\alpha,1}(x,y)\big)^{2}\,d\z\bigg)^{1\slash 2}\\
\lesssim&
\bigg(\int_{0}^{1}\lo(\z)\frac{1}{1-\z^2}x_{1}^{2}y_{1}^{2}\Big(\frac{1-\z^2}{2\z}\Big)^{2d+2|\a|+2}\Big(\frac{1}{\z}\Big)^{2}
\bigg(\int\big(\e\b\big)^{1\slash 2}\,\Pi_{\a+e_{1}}(ds)\bigg)^{2}\,d\z\bigg)^{1\slash 2}\\
\leq&
\int\bigg(\int_{0}^{1}\lo(\z)x_{1}^{2}y_{1}^{2}\Big(\frac{1}{\z}\Big)^{2d+2|\a|+4}
\e\b\,d\z\bigg)^{1\slash 2}\,\Pi_{\a+e_{1}}(ds)\\
\lesssim&
\,x_{1}y_{1}\int(q_{+})^{-d-|\a|-1\slash 2}\bigg(\int_{0}^{1}\lo(\z)\Big(\frac{1}{\z}\Big)^{3}\big(\e\b\big)^{1\slash 2}\,d\z\bigg)^{1\slash 2}\,\Pi_{\a+e_{1}}(ds)\\
\lesssim&
(x+y)^{2e_{1}}\int(q_{+})^{-d-|\a|-1}\,\Pi_{\a+e_{1}}(ds).
\end{align*}
Now the growth estimate follows by Lemma \ref{lem4} (taken with $\delta=e_{1}$ and $\kappa=0$).

It remains to prove the smoothness estimates. By symmetry reasons, it suffices to show that
$$
\big\|\partial_{t}\widetilde G_{t}^{\a,1}(x,y)-\partial_{t}\widetilde G_{t}^{\a,1}(x',y)\big\|_{L^2(tdt)}
\lesssim
\frac{|x-x'|}{|x-y|} \;
\frac{1}{\mu_{\alpha}(B(x,|y-x|))},\qquad |x-y|>2|x-x'|.
$$
Define (see \eqref{nier3})
$$
\varphi(x,y,\z)=\partial_{t}\widetilde G_{t}^{\a,1}(x,y)\slash x_{1}.
$$
Then by the mean value theorem
\begin{align*}
\big|\partial_{t}\widetilde G_{t}^{\a,1}(x,y)-\partial_{t}\widetilde G_{t}^{\a,1}(x',y)\big|
=&
\,|x_{1}\v(x,y,\z)-x_{1}'\v(x',y,\z)|\\
\leq&
\,\big|(x_{1}-x_{1}')\v(x',y,\z)\big|
+
\big|x_{1}\big(\v(x,y,\z)-\v(x',y,\z)\big)\big|\\
\leq&
\,|x-x'||\v(x',y,\z)|
+
x_{1}|x-x'|\big|\nabla_{\!x}\v(\t,y,\z)\big|,
\end{align*}
where $\t$ is a convex combination of $x$ and $x'$.
We will treat separately each of the two terms in the last expression.

In view of estimates that already appeared in this proof and Lemma \ref{lemat} we have
\begin{align*}
\big|\v(x',y,\z)\big|
\lesssim&
\,y_{1}\Big(\frac{1-\z^2}{2\z}\Big)^{d+|\a|+1}\z^{-1}\int\big[\e\ccc\big]^{1\slash 2}\,\Pi_{\a+e_{1}}(ds)\\
\leq&
\,y_{1}\sqrt{1-\z^2}\,\z^{-d-|\a|-2}\int\big[\e\bbb\big]^{1\slash 8}\,\Pi_{\a+e_{1}}(ds),
\end{align*}
provided that $|x-y|>2|x-x'|$. Proceeding as in the case of the growth condition we obtain
\begin{align*}
\big\|\v(x',y,\z(t))\big\|_{L^2(tdt)}
\lesssim
(x+y)^{e_{1}}\int(q_{+})^{-d-|\a|-1}\,\Pi_{\a+e_{1}}(ds).
\end{align*}
Then Lemma \ref{lem4} (taken with $\delta=e_{1}\slash 2$ and $\kappa=e_{1}\slash 2$) shows that the quantity $|x-x'||\v(x',y,\z)|$ satisfies the smoothness estimate, as desired.

To finish the proof we must verify the same for the quantity $x_{1}|x-x'|\big|\nabla_{\!x}\v(\t,y,\z)\big|$, and this boils down to showing that for any $i=1,\ldots,d$
\begin{align*}
\big\|x_{1}\partial_{x_{i}}\v(\t,y,\z(t))\big\|_{L^2(tdt)}
\lesssim
\frac{1}{|x-y|} \;
\frac{1}{\mu_{\alpha}(B(x,|y-x|))},\qquad |x-y|>2|x-x'|.
\end{align*}
Observe that (see \eqref{nier3})
\begin{align*}
\partial_{x_{i}}\v(\t,y,\z)=
e^{-2t}y_{1}\big(-2\partial_{x_{i}}G_{t}^{\a+e_{1}}(\t,y)+\partial_{x_{i}}\partial_{t}G_{t}^{\a+e_{1}}(\t,y)\big).
\end{align*}
Using the estimate of $\partial_{x_{i}}G_{t}^{\a}(x,y)$ given implicitly in Section \ref{k2} and the estimate of $\partial_{x_{i}}\partial_{t}G_{t}^{\a}(x,y)$ that can be read off from Section \ref{k1}, and then Lemma \ref{lemat}, we get
\begin{align*}
x_{1}\big|\partial_{x_{i}}\v(\t,y,\z)\big|
\lesssim&
\,x_{1}y_{1}\Big(\frac{1-\z^2}{2\z}\Big)^{d+|\a|+1}\z^{-1\slash 2}\int\big[\e\c\big]^{1\slash 2}\,\Pi_{\a+e_{1}}(ds)\\
&+
\,x_{1}y_{1}\Big(\frac{1-\z^2}{2\z}\Big)^{d+|\a|+1}\z^{-3\slash 2}\int\big[\e\c\big]^{1\slash 4}\,\Pi_{\a+e_{1}}(ds)\\
\lesssim&
\,x_{1}y_{1}\sqrt{1-\z^2}\,\z^{-d-|\a|-5\slash 2}\int\big[\e\bbb\big]^{1\slash 16}\,\Pi_{\a+e_{1}}(ds),
\end{align*}
provided that $|x-y|>2|x-x'|$.
Changing the variable as in \eqref{zz} and then using sequently the above estimate, the Minkowski inequality, Lemma \ref{oq} (applied with $b=2d+2|\a|+2$, $c=1\slash 32$, $A=\z^{-1}$) and Lemma \ref{uw} (with $C=1\slash 64$) we obtain
\begin{align*}
\big\|x_{1}\partial_{x_{i}}&\v(\t,y,\z(t))\big\|_{L^2(tdt)}\\
\lesssim&
\bigg(\int\lo(\z)x_{1}^{2}y_{1}^{2}\Big(\frac{1}{\z}\Big)^{2d+2|\a|+5}
\bigg(\int\big(\e\b\big)^{1\slash 16}\,\Pi_{\a+e_{1}}(ds)\bigg)^{2}\,d\z\bigg)^{1\slash 2}\\
\leq&
\int\bigg(\int_{0}^{1}\lo(\z)x_{1}^{2}y_{1}^{2}\Big(\frac{1}{\z}\Big)^{2d+2|\a|+5}
\big(\e\b\big)^{1\slash 8}\,d\z\bigg)^{1\slash 2}\,\Pi_{\a+e_{1}}(ds)\\
\lesssim&
\,x_{1}y_{1}\int(q_{+})^{-d-|\a|-1}\bigg(\int_{0}^{1}\lo(\z)\Big(\frac{1}{\z}\Big)^{3}\big(\e\b\big)^{1\slash 16}\,d\z\bigg)^{1\slash 2}\,\Pi_{\a+e_{1}}(ds)\\
\lesssim&
(x+y)^{2e_{1}}\int(q_{+})^{-d-|\a|-3\slash 2}\,\Pi_{\a+e_{1}}(ds).
\end{align*}
Now an application of Lemma \ref{lem4} (with $\delta=e_{1}$ and $\kappa=0$) leads to the desired conclusion. 

The proof of the case of $g_{V,\widetilde T}^{j}$ in Theorem \ref{preg} is complete.
\end {proof}
\subsection{Horizontal $g$-functions based on  $\mathbf{\{\widetilde T_{t}^{\a,j}\}}$}\label{k5}

\begin {proof}[Proof of Theorem \ref{preg}; the case of $g_{H,\widetilde T}^{j,i}$, $j\ne i$.]
By symmetry reasons we may focus on the case $j=1,$ $i=2$. With the aid of \eqref{falka} and Lemma \ref{wch} we get
\begin{equation}\label{nier4}
\delta_{2,x}\widetilde G_{t}^{\alpha,1}(x,y)=
\,x_{1}y_{1}\frac{1-\z}{1+\z}\Big(\frac{1-\z^2}{2\z}\Big)^{d+|\a|+1}h(x,y,\z),
\end{equation}
where the auxiliary function $h$ is now given by
\begin{align*}
h(x,y,\z)=&
-\int\e\b\Big[\frac{1}{2\z}(x_{2}+y_{2}s_{2})+\frac{\z}{2}(x_{2}-y_{2}s_{2})\Big]\,\Pi_{\a+e_{1}}(ds)\\
&+
x_{2}\int\e\b\,\Pi_{\a+e_{1}}(ds).
\end{align*}
Using Lemma \ref{obs}, the fact that $x_{2}\leq \sqrt{q_{+}}+\sqrt{q_{-}}$ and then Lemma \ref{oq} (with $b=1\slash 2$, $A=\z^{-1}$ and $A=\z$, respectively) we see that
\begin{align}
|h(x,y,\z)|\leq&
\int\e\b\Big[\frac{1}{\z}\sqrt{q_{+}}+\z\sqrt{q_{-}}\Big]\Pi_{\a+e_{1}}(ds)
+\int(\sqrt{q_{+}}+\sqrt{q_{-}})\e\b\Pi_{\a+e_{1}}(ds)\nonumber\\
\lesssim&
\,\z^{-1\slash 2}\int\big(\e\b\big)^{1\slash 2}\Pi_{\a+e_{1}}(ds)\label{nier}.
\end{align}
Now changing the variable according to \eqref{zz} and then using sequently the above estimate, the Minkowski inequality, Lemma \ref{oq} (taken with $b=2d+2|\alpha|+1$, $c=1\slash 4$, $A=\zeta^{-1}$) and Lemma \ref{lem5.4} (with $a=2$ and $T=\frac{1}{8}q_{+}$) we obtain
\begin{align*}
\big\|\delta_{2,x}\widetilde G_{t}^{\alpha,1}(x,y)\big\|_{L^{2}(dt)}=&
\bigg(\int_{0}^{1}\frac{1}{1-\z^2}x_{1}^2y_{1}^2\Big(\frac{1-\z}{1+\z}\Big)^{2}\Big(\frac{1-\z^2}{2\z}\Big)^{2d+2|\a|+2}
|h(x,y,\z)|^2\,d\z \bigg)^{1\slash 2}\\
\lesssim&
\,x_{1}y_{1}\bigg(\int_{0}^{1}\Big(\frac{1}{\z}\Big)^{2d+2|\a|+2}\frac{1}{\z}\bigg(
\int\big(\e\b\big)^{1\slash 2}\,\Pi_{\a+e_{1}}\bigg)^{2}\,d\z\bigg)^{1\slash 2}\\
\leq&
\,x_{1}y_{1}\int\bigg(\int_{0}^{1}\Big(\frac{1}{\z}\Big)^{2d+2|\a|+1}
\Big(\frac{1}{\z}\Big)^{2}\e\b\,d\z\bigg)^{1\slash 2}\,\Pi_{\a+e_{1}}(ds)\\
\lesssim&
(x+y)^{2e_{1}}\int(q_{+})^{-d-|\a|-1\slash 2}\bigg(\int_{0}^{1}\Big(\frac{1}{\z}\Big)^{2}
\big(\e\b\big)^{1\slash 2}\,d\z\bigg)^{1\slash 2}\,\Pi_{\a+e_{1}}(ds)\\
\lesssim&
(x+y)^{2e_{1}}\int(q_{+})^{-d-|\a|-1}\,\Pi_{\a+e_{1}}(ds).
\end{align*}
In view of Lemma \ref{lem4} (applied with $\delta=e_{1}$ and $\kappa=0$) the growth estimate follows.

To prove the smoothness conditions we first show that
$$
\big\|\delta_{2,x}\widetilde G_{t}^{\alpha,1}(x,y)-
\delta_{2,x}\widetilde G_{t}^{\alpha,1}(x',y)\big\|_{L^2(dt)}\lesssim
\frac{|x-x'|}{|x-y|} \;
\frac{1}{\mu_{\alpha}(B(x,|y-x|))},\qquad |x-y|>2|x-x'|.
$$ 
Define (see \eqref{nier4})
$$
\v(x,y,\z)=\delta_{2,x}\widetilde G_{t}^{\alpha,1}(x,y)\slash x_{1}.
$$
Then by the mean value theorem we have
\begin{align*}
\big|\delta_{2,x}\widetilde G_{t}^{\alpha,1}(x,y)-
\delta_{2,x}\widetilde G_{t}^{\alpha,1}(x',y)\big|
\leq
|x-x'||\v(x',y,\z)|
+
x_{1}|x-x'|\big|\nabla_{\!x}\v(\t,y,\z)\big|.
\end{align*}
We will treat separately each of the two terms in the right-hand side above.

Using the estimate \eqref{nier} and then Lemma \ref{lemat} we get
\begin{align}
\big|\v(x',y,\z)\big|
\lesssim&
\,y_{1}\Big(\frac{1-\z^2}{2\z}\Big)^{d+|\a|+1}
\z^{-1\slash 2}\int\big[\e\ccc\big]^{1\slash 2}\Pi_{\a+e_{1}}(ds)\nonumber\\
\leq&
\,(x_{1}+y_{1})\sqrt{1-\z^2}\,\z^{-d-|\a|-3\slash 2}\int\big[\e\bbb\big]^{1\slash 8}\Pi_{\a+e_{1}}(ds),\label{nier5}
\end{align}
provided that $|x-y|>2|x-x'|$.
Proceeding in a similar way as in the first part of the proof we obtain
\begin{align*}
\big\|\v(x',y,\z)\big\|_{L^2(dt)}
\lesssim&
(x+y)^{e_{1}}\int(q_{+})^{-d-|\a|-1}\,\Pi_{\a+e_{1}}(ds),
\end{align*}
which in view of Lemma \ref{lem4} (taken with $\delta=e_{1}\slash 2$ and $\kappa=e_{1}\slash 2$) gives the relevant bound of $|x-x'||\v(x',y,\z)|$. To bound suitably $x_{1}|x-x'|\big|\nabla_{\!x}\v(\t,y,\z)\big|$ it suffices to show that for any $k=1,\ldots,d,$
\begin{align*}
\big\|x_{1}\partial_{x_{k}}\v(\t,y,\z(t))\big\|_{L^2(dt)}
\lesssim
\frac{1}{|x-y|} \;
\frac{1}{\mu_{\alpha}(B(x,|y-x|))},\qquad |x-y|>2|x-x'|.
\end{align*}

We have
\begin{align*}
\partial_{x_{k}}\v(x,y,\z)=&
\,y_{1}\frac{1-\z}{1+\z}\Big(\frac{1-\z^2}{2\z}\Big)^{d+|\a|+1}\partial_{x_{k}}h(x,y,\z),\qquad k=1,\ldots,d.
\end{align*}
While estimating $\partial_{x_{k}}h(x,y,\z)$ it is convenient to distinguish two cases.
\newline
{\bf Case 1:} $\mathbf {k\ne 2.}$ We have
\begin{align*}
\partial_{x_{k}}&h(x,y,\z)\\
=&
\int\e\b\Big[\frac{1}{2\z}(x_{k}+y_{k}s_{k})+\frac{\z}{2}(x_{k}-y_{k}s_{k})\Big]
\Big[\frac{1}{2\z}(x_{2}+y_{2}s_{2})+\frac{\z}{2}(x_{2}-y_{2}s_{2})\Big]\,\Pi_{\a+e_{1}}(ds)\\
&-
x_{2}\int\e\b\Big[\frac{1}{2\z}(x_{k}+y_{k}s_{k})+\frac{\z}{2}(x_{k}-y_{k}s_{k})\Big]
\,\Pi_{\a+e_{1}}(ds).
\end{align*}
Using Lemma \ref{obs}, the fact that $x_{2}\leq\sqrt{q_{+}}+\sqrt{q_{-}}$ and then Lemma \ref{oq} (specified to $b=1\slash 2$, $A=\zeta^{-1}$ and $A=\zeta$, respectively) we obtain
\begin{align*}
|\partial_{x_{k}}h(x,y,\z)|\leq&
\int\e\b\Big[\frac{1}{\z}\sqrt{q_{+}}+\z\sqrt{q_{-}}\Big]^2\,\Pi_{\a+e_{1}}(ds)\\
&+
\int(\sqrt{q_{+}}+\sqrt{q_{-}})\e\b\Big[\frac{1}{\z}\sqrt{q_{+}}+\z\sqrt{q_{-}}\Big]\,\Pi_{\a+e_{1}}(ds)\\
\lesssim&
\,\z^{-1}\int\big(\e\b\big)^{1\slash 4}\,\Pi_{\a+e_{1}}(ds).
\end{align*}
{\bf Case 2:} $\mathbf {k=2.}$
An elementary computation shows that
\begin{align*}
\partial_{x_{2}}h(x,y,\z)=&
\int\e\b\Big[\frac{1}{2\z}(x_{2}+y_{2}s_{2})+\frac{\z}{2}(x_{2}-y_{2}s_{2})\Big]^{2}\,\Pi_{\a+e_{1}}(ds)\\
&-
\int\e\b\Big[\frac{1}{2\z}+\frac{\z}{2}\Big]\,\Pi_{\a+e_{1}}(ds)
+
\int\e\b\,\Pi_{\a+e_{1}}(ds)\\
&-
x_{2}\int\e\b\Big[\frac{1}{2\z}(x_{2}+y_{2}s_{2})+\frac{\z}{2}(x_{2}-y_{2}s_{2})\Big]\,\Pi_{\a+e_{1}}(ds).
\end{align*}
Proceeding similarly as in Case 1 we get
$$
|\partial_{x_{2}}h(x,y,\z)|\lesssim
\z^{-1}\int\big(\e\b\big)^{1\slash 4}\,\Pi_{\a+e_{1}}(ds).
$$ 
Using the above estimates of $\partial_{x_{k}}h(x,y,\z)$ and Lemma \ref{lemat} we see that
\begin{align}
x_{1}\big|\partial_{x_{k}}\v(\t,y,\z(t))\big|
\lesssim&
\,x_{1}y_{1}\Big(\frac{1-\z^2}{2\z}\Big)^{d+|\a|+1}
\z^{-1}\int\big[\e\c\big]^{1\slash 4}\,\Pi_{\a+e_{1}}(ds)\nonumber\\
\leq&
\,x_{1}y_{1}\sqrt{1-\z^2}\,\z^{-d-|\a|-2}\int\big[\e\bbb\big]^{1\slash 16}\,\Pi_{\a+e_{1}}(ds),\label{nier6}
\end{align}
provided that $|x-y|>2|x-x'|$. Changing the variable according to \eqref{zz} and then applying the Minkowski inequality, Lemma \ref{oq} (specified to $b=2d+2|\alpha|+2$, $c=1\slash 32$, $A=\zeta^{-1}$) and Lemma \ref{lem5.4} (with $a=2$ and $T=\frac{1}{64}q_{+}$) we get
\begin{align*}
\big\|x_{1}\partial_{x_{k}}&\v(\t,y,\z(t))\big\|_{L^2(dt)}\\
\lesssim&
\bigg(\int_{0}^{1}x_{1}^2y_{1}^2
\Big(\frac{1}{\z}\Big)^{2d+2|\a|+4}\bigg(\int\big(\e\b\big)^{1\slash 16}\,\Pi_{\a+e_{1}}(ds)\bigg)^{2}\,d\z \bigg)^{1\slash 2}\\
\leq&
\,x_{1}y_{1}\int\bigg(\int_{0}^{1}\Big(\frac{1}{\z}\Big)^{2d+2|\a|+4}
\big(\e\b\big)^{1\slash 8}\,d\z\bigg)^{1\slash 2}\,\Pi_{\a+e_{1}}(ds)\\
\lesssim&
(x+y)^{2e_{1}}\int(q_{+})^{-d-|\a|-1}\bigg(\int_{0}^{1}\Big(\frac{1}{\z}\Big)^{2}
\big(\e\b\big)^{1\slash 16}\,d\z\bigg)^{1\slash 2}\,\Pi_{\a+e_{1}}(ds)\\
\lesssim&
(x+y)^{2e_{1}}\int(q_{+})^{-d-|\a|-3\slash 2}\,\Pi_{\a+e_{1}}(ds).
\end{align*}
Now Lemma \ref{lem4} (taken with $\delta=e_{1}$ and $\kappa=0$) implies the required bound.

The proof will be finished once we show that
$$
\big\|\delta_{2,x}\widetilde G_{t}^{\alpha,1}(x,y)-
\delta_{2,x}\widetilde G_{t}^{\alpha,1}(x,y')\big\|_{L^2(dt)}\lesssim
\frac{|y-y'|}{|x-y|} \;
\frac{1}{\mu_{\alpha}(B(x,|y-x|))},\qquad |x-y|>2|y-y'|.
$$ 
Define (see \eqref{nier4})
$$
\phi(x,y,\z)=\delta_{2,x}\widetilde G_{t}^{\alpha,1}(x,y)\slash y_{1}.
$$
By the mean value theorem we have
\begin{align*}
\big|\delta_{2,x}\widetilde G_{t}^{\alpha,1}(x,y)-
\delta_{2,x}\widetilde G_{t}^{\alpha,1}(x,y')\big|
\leq
|y-y'||\phi(x,y',\z)|
+
y_{1}|y-y'|\big|\nabla_{\!y}\phi(x,\t,\z)\big|,
\end{align*}
where $\t$ is a convex combination of $y$ and $y'$.
Parallel arguments to those used in the proof of the growth condition and Lemma \ref{lemat} lead to
\begin{align*}
\big\|\phi(x,y',\z(t))\big\|_{L^2(dt)}
\lesssim
\frac{1}{|x-y|} \;
\frac{1}{\mu_{\alpha}(B(x,|y-x|))},\qquad |x-y|>2|y-y'|.
\end{align*}
This implies the desired estimate of the term $|y-y'||\phi(x,y',\z)|$. To bound suitably the remaining term $y_{1}|y-y'|\big|\nabla_{\!y}\phi(x,\t,\z)\big|$ it is enough, in view of the above considerations, to verify that
$$
|\partial_{y_{k}}h(x,y,\z)|\lesssim
\z^{-1}\int\big(\e(\z,q_{\pm})\big)^{1\slash 4}\,\Pi_{\a+e_{1}}(ds),\qquad  k=1,\ldots,d,
$$
since
\begin{align*}
\partial_{y_{k}}\phi(x,y,\z)=&
\,x_{1}\frac{1-\z}{1+\z}\Big(\frac{1-\z^2}{2\z}\Big)^{d+|\a|+1}\partial_{y_{k}}h(x,y,\z),\qquad k=1,\ldots,d.
\end{align*}
Again, it is natural to distinguish two cases.
\newline
{\bf Case 1:} $\mathbf {k\ne 2.}$ 
We have
\begin{align*}
\partial_{y_{k}}&h(x,y,\z)\\
=&
\int\e\b\Big[\frac{1}{2\z}(y_{k}+x_{k}s_{k})+\frac{\z}{2}(y_{k}-x_{k}s_{k})\Big]
\Big[\frac{1}{2\z}(x_{2}+y_{2}s_{2})+\frac{\z}{2}(x_{2}-y_{2}s_{2})\Big]\,\Pi_{\a+e_{1}}(ds)\\
&-
x_{2}\int\e\b\Big[\frac{1}{2\z}(y_{k}+x_{k}s_{k})+\frac{\z}{2}(y_{k}-x_{k}s_{k})\Big]
\,\Pi_{\a+e_{1}}(ds).
\end{align*}
Now it is not hard to check that (see the estimate of $\partial_{x_{k}}h(x,y,\z)$ above)
$$
|\partial_{y_{k}}h(x,y,\z)|\lesssim
\z^{-1}\int\big(\e\b\big)^{1\slash 4}\,\Pi_{\a+e_{1}}(ds).
$$
{\bf Case 2:} $\mathbf {k=2.}$ 
An elementary computation produces
\begin{align*}
\partial_{y_{2}}&h(x,y,\z)\\
=&
\int\e\b\Big[\frac{1}{2\z}(y_{2}+x_{2}s_{2})+\frac{\z}{2}(y_{2}-x_{2}s_{2})\Big]
\Big[\frac{1}{2\z}(x_{2}+y_{2}s_{2})+\frac{\z}{2}(x_{2}-y_{2}s_{2})\Big]\,\Pi_{\a+e_{1}}(ds)\\
&-
\int\e\b\Big[\frac{1}{2\z}s_{2}-\frac{\z}{2}s_{2}\Big]\,\Pi_{\a+e_{1}}(ds)\\
&-
x_{2}\int\e\b\Big[\frac{1}{2\z}(y_{2}+x_{2}s_{2})+\frac{\z}{2}(y_{2}-x_{2}s_{2})\Big]
\,\Pi_{\a+e_{1}}(ds).
\end{align*}
Parallel arguments to those from Case 1 lead to
$$
|\partial_{y_{2}}h(x,y,\z)|\lesssim
\z^{-1}\int\big(\e\b\big)^{1\slash 4}\,\Pi_{\a+e_{1}}(ds).
$$
This finishes proving the case of $g_{H,\widetilde T}^{j,i}$, $j\ne i$, in Theorem \ref{preg}.
\end {proof}

\begin {proof}[Proof of Theorem \ref{preg}; the case of $g_{H,\widetilde T}^{j,j}$.]
With no loss of generality we may focus only on the case $j=1$. Recall that $\delta_{1,x}^{*}=-\partial_{x_{1}}+x_{1}-\frac{2\a_{1}+1}{x_{1}}$. In view of \eqref{falka} and Lemma \ref{wch} we get
\begin{align}\label{row}
\delta_{1,x}^{*}\widetilde G_{t}^{\alpha,1}(x,y)=
y_{1}\frac{1-\z}{1+\z}\Big(\frac{1-\z^2}{2\z}\Big)^{d+|\a|+1}h(x,y,\z),
\end{align}
where the auxiliary function $h$ is now given by
\begin{align*}
h(x,y,\z)=&
-(2\a_{1}+2)\int\e\b\,\Pi_{\a+e_{1}}(ds)\\
&+
x_{1}\int\e\b\Big[\frac{1}{2\z}(x_{1}+y_{1}s_{1})+\frac{\z}{2}(x_{1}-y_{1}s_{1})\Big]\,\Pi_{\a+e_{1}}(ds)\\
&+
x_{1}^{2}\int\e\b\,\Pi_{\a+e_{1}}(ds).
\end{align*}
Using Lemma \ref{obs}, the fact that $x_{1}\leq \sqrt{q_{+}}+\sqrt{q_{-}}$, and then Lemma \ref{oq} (with $b=1\slash 2$, $A=\z^{-1}$ and $A=\z$, respectively) we obtain
\begin{align}
|h(x,y,\z)|\lesssim&
\int\e\b\,\Pi_{\a+e_{1}}(ds)
+
x_{1}\int\e\b\Big[\frac{1}{\z}\sqrt{q_{+}}+\z\sqrt{q_{-}}\Big]\,\Pi_{\a+e_{1}}(ds)\nonumber\\
&+
x_{1}\int(\sqrt{q_{+}}+\sqrt{q_{-}})\e\b\,\Pi_{\a+e_{1}}(ds)\nonumber\\
\lesssim&
\int\e\b\,\Pi_{\a+e_{1}}(ds)
+
\,x_{1}\z^{-1\slash 2}\int\big(\e\b\big)^{1\slash 2}\,\Pi_{\a+e_{1}}(ds).\label{nier2}
\end{align}
Applying this estimate we get
\begin{align*}
\big|\delta_{1,x}^{*}\widetilde G_{t}^{\alpha,1}(x,y)\big|
\lesssim&
\,y_{1}\Big(\frac{1-\z^2}{2\z}\Big)^{d+|\a|+1}\int\e\b\,\Pi_{\a+e_{1}}(ds)\\
&+
\,x_{1}y_{1}\Big(\frac{1-\z^2}{2\z}\Big)^{d+|\a|+1}\z^{-1\slash 2}\int\big(\e\b\big)^{1\slash 2}\,\Pi_{\a+e_{1}}(ds)\\
\equiv&
\,h_{1}(x,y,\z)+h_{2}(x,y,\z).
\end{align*}

We will treat $h_{1}$ and $h_{2}$ separately.
Changing the variable according to \eqref{zz} and then using the Minkowski inequality, Lemma \ref{oq} (taken with $b=2d+2|\alpha|$, $c=1\slash 2$, $A=\zeta^{-1}$) and Lemma \ref{lem5.4} (with $a=2$ and $T=\frac{1}{4}q_{+}$) we obtain
\begin{align*}
\|h_{1}(x,y,\z(t))\|_{L^2(dt)}=&
\bigg(\int_{0}^{1}\frac{1}{1-\z^2}y_{1}^{2}\Big(\frac{1-\z^2}{2\z}\Big)^{2d+2|\a|+2}
\bigg(\int\e\b\,\Pi_{\a+e_{1}}(ds)\bigg)^2\,d\z\bigg)^{1\slash 2}\\
\leq&
\,y_{1}\int\bigg(\int_{0}^{1}\Big(\frac{1}{\z}\Big)^{2d+2|\a|}
\Big(\frac{1}{\z}\Big)^{2}\big(\e\b\big)^{2}\,d\z\bigg)^{1\slash 2}\,\Pi_{\a+e_{1}}(ds)\\
\lesssim&
\,y_{1}\int(q_{+})^{-d-|\a|}\bigg(\int_{0}^{1}\Big(\frac{1}{\z}\Big)^{2}\e\b\,d\z\bigg)^{1\slash 2}\,\Pi_{\a+e_{1}}(ds)\\
\lesssim&
(x+y)^{e_{1}}\int(q_{+})^{-d-|\a|-1\slash 2}\,\Pi_{\a+e_{1}}(ds).
\end{align*}
Now the growth estimate for $h_{1}$ follows with the aid of Lemma \ref{lem4} (specified to $\delta=e_{1}\slash 2$ and $\kappa=e_{1}\slash 2$).
The growth condition for $h_{2}$ was proved implicitly earlier, see \eqref{nier} and the succeeding estimates. The growth estimate for $\big\{\delta_{1,x}^{*}\widetilde G_{t}^{\a,1}(x,y)\big\}$ follows.

To prove the smoothness conditions we first show that 
$$
\big\|\delta_{1,x}^{*}\widetilde G_{t}^{\alpha,1}(x,y)-\delta_{1,x}^{*}\widetilde G_{t}^{\alpha,1}(x',y)\big\|_{L^2(dt)}\lesssim
\frac{|x-x'|}{|x-y|} \;
\frac{1}{\mu_{\alpha}(B(x,|y-x|))},\qquad |x-y|>2|x-x'|.
$$
Define the auxiliary functions 
\begin{align*}
\v_{1}(x,y,\z)=&
-(2\a_{1}+2)\int\e\b\,\Pi_{\a+e_{1}}(ds),\\
\v_{2}(x,y,\z)=&
\int\e\b\Big[\frac{1}{2\z}(x_{1}+y_{1}s_{1})+\frac{\z}{2}(x_{1}-y_{1}s_{1})\Big]\,\Pi_{\a+e_{1}}(ds)\\
&+
x_{1}\int\e\b\,\Pi_{\a+e_{1}}(ds),
\end{align*}
so that (see \eqref{row})
\begin{align*}
\delta_{1,x}^{*}\widetilde G_{t}^{\alpha,1}(x,y)
=
y_{1}\frac{1-\z}{1+\z}\Big(\frac{1-\z^2}{2\z}\Big)^{d+|\a|+1}
\big(\v_{1}(x,y,\z)+x_{1}\v_{2}(x,y,\z)\big).
\end{align*}
By the mean value theorem
\begin{align*}
\big|\delta_{1,x}^{*}&\widetilde G_{t}^{\alpha,1}(x,y)-\delta_{1,x}^{*}\widetilde G_{t}^{\alpha,1}(x',y)\big|\\
\leq&
\,y_{1}\frac{1-\z}{1+\z}\Big(\frac{1-\z^2}{2\z}\Big)^{d+|\a|+1}|x-x'|
\Big(\big|\nabla_{\!x}\v_{1}(\t_{1},y,\z)\big|+|\v_{2}(x',y,\z)|+x_{1}\big|
\nabla_{\!x}\v_{2}(\t_{2},y,\z)\big|\Big),
\end{align*}
where $\t_{1}$, $\t_{2}$ are convex combinations of $x$ and $x'$. We now analyze each of the three terms above.
For any $i=1,\ldots,d$ we have 
\begin{align*}
\partial_{x_{i}}\v_{1}(x,y,\z)=
(2\a_{1}+2)\int\e\b\Big[\frac{1}{2\z}(x_{i}+y_{i}s_{i})+\frac{\z}{2}(x_{i}-y_{i}s_{i})\Big]
\,\Pi_{\a+e_{1}}(ds).
\end{align*}
Using Lemma \ref{obs} and then Lemma \ref{oq} (with $b=1\slash 2$, $A=\z^{-1}$ and $A=\z$, respectively) we get
\begin{align*}
|\partial_{x_{i}}\v_{1}(x,y,\z)|
\lesssim&
\int\e\b\Big[\frac{1}{\z}\sqrt{q_{+}}+\z\sqrt{ q_{-}}\Big]\,\Pi_{\a+e_{1}}(ds)\\
\lesssim&
\,\z^{-1\slash 2}\int\big(\e\b\big)^{1\slash 2}\,\Pi_{\a+e_{1}}(ds).
\end{align*}
This together with Lemma \ref{lemat} gives 
\begin{align}\label{oszac1}
|\nabla_{\!x}\v_{1}(\t_{1},y,\z)|
\lesssim&
\,\z^{-1\slash 2}\int\big[\e\bbb\big]^{1\slash 8}\,\Pi_{\a+e_{1}}(ds)
\end{align}
for $|x-y|>2|x-x'|$. Further, it is not hard to check that
\begin{align*}
|\v_{2}(x,y,\z)|
\lesssim
\z^{-1\slash 2}\int\big(\e\b\big)^{1\slash 2}\,\Pi_{\a+e_{1}}(ds),
\end{align*}
so using Lemma \ref{lemat} we obtain the same bound as before,
\begin{align}\label{oszac2}
|\v_{2}(x',y,\z)|
\lesssim
\z^{-1\slash 2}\int\big[\e\bbb\big]^{1\slash 8}\,\Pi_{\a+e_{1}}(ds),
\end{align}
provided that $|x-y|>2|x-x'|$. Finally, estimating $x_{1}\big|\nabla_{\!x}\v_{2}(\t_{2},y,\z)\big|$ requires proper bounds on $\partial_{x_{i}}\v_{2}(x,y,\z)$. It is convenient to distinguish two cases.
\newline
{\bf Case 1:} $\mathbf {i\ne 1.}$
By symmetry reasons, we may assume that $i=2$. Then
\begin{align*}
\partial_{x_{2}}&\v_{2}(x,y,\z)\\
=&
-
\int\e\b\Big[\frac{1}{2\z}(x_{2}+y_{2}s_{2})+\frac{\z}{2}(x_{2}-y_{2}s_{2})\Big]
\Big[\frac{1}{2\z}(x_{1}+y_{1}s_{1})+\frac{\z}{2}(x_{1}-y_{1}s_{1})\Big]\,\Pi_{\a+e_{1}}(ds)\\
&-
x_{1}\int\e\b\Big[\frac{1}{2\z}(x_{2}+y_{2}s_{2})+\frac{\z}{2}(x_{2}-y_{2}s_{2})\Big]
\,\Pi_{\a+e_{1}}(ds).
\end{align*}
Using Lemma \ref{obs}, the fact that $x_{1}\leq \sqrt{q_{+}}+\sqrt{q_{-}}$ and then Lemma \ref{oq} (with $b=1\slash 2$, $A=\z^{-1}$ and $A=\z$, respectively) we get
\begin{align*}
|\partial_{x_{2}}\v_{2}(x,y,\z)|
\lesssim&
\int\e\b\Big[\frac{1}{\z}\sqrt{q_{+}}+\z \sqrt{q_{-}}\Big]^{2}\,\Pi_{\a+e_{1}}(ds)\\
&+
\int(\sqrt{q_{+}}+\sqrt{q_{-}})\e\b\Big[\frac{1}{\z}\sqrt{q_{+}}+\z \sqrt{q_{-}}\Big]\,\Pi_{\a+e_{1}}(ds)\\
\lesssim&
\,\z^{-1}\int\big(\e\b\big)^{1\slash 4}\,\Pi_{\a+e_{1}}(ds).
\end{align*}
{\bf Case 2:} $\mathbf {i=1.}$
An elementary computation gives
\begin{align*}
\partial_{x_{1}}\v_{2}(x,y,\z)=&
-
\int\e\b\Big[\frac{1}{2\z}(x_{1}+y_{1}s_{1})+\frac{\z}{2}(x_{1}-y_{1}s_{1})\Big]^{2}
\,\Pi_{\a+e_{1}}(ds)\\
&+
\int\e\b\Big[\frac{1}{2\z}+\frac{\z}{2}\Big]\,\Pi_{\a+e_{1}}(ds)
+
\int\e\b\,\Pi_{\a+e_{1}}(ds)\\
&-
x_{1}\int\e\b\Big[\frac{1}{2\z}(x_{1}+y_{1}s_{1})+\frac{\z}{2}(x_{1}-y_{1}s_{1})\Big]
\,\Pi_{\a+e_{1}}(ds).
\end{align*}
Proceeding in a similar way as in Case 1 we obtain the same bound as before,
\begin{align*}
|\partial_{x_{1}}\v_{2}(x,y,\z)|
\lesssim&
\,\z^{-1}\int\big(\e\b\big)^{1\slash 4}\,\Pi_{\a+e_{1}}(ds).
\end{align*}
By the above estimates of $\partial_{x_{i}}\v_{2}(x,y,\z)$ and Lemma \ref{lemat} we see that
\begin{align}\label{oszac3}
|\nabla_{\!x}\v_{2}(\t_{2},y,\z)|
\lesssim&
\,\z^{-1}\int\big[\e\bbb\big]^{1\slash 16}\,\Pi_{\a+e_{1}}(ds)
\end{align}
for $|x-y|>2|x-x'|$. Taking into account \eqref{oszac1}, \eqref{oszac2} and \eqref{oszac3} we obtain, for $|x-y|>2|x-x'|$,
\begin{align*}
\big|\delta_{1,x}^{*}&\widetilde G_{t}^{\alpha,1}(x,y)-\delta_{1,x}^{*}\widetilde G_{t}^{\alpha,1}(x',y)\big|\\
\lesssim&
|x-x'|(x_{1}+y_{1})\sqrt{1-\z^2}\,\z^{-d-|\a|-3\slash 2}
\int\big(\e\b\big)^{1\slash 8}\,\Pi_{\a+e_{1}}(ds)\\
&+
|x-x'|\,x_{1}y_{1}\sqrt{1-\z^2}\,\z^{-d-|\a|-2}
\int\big(\e\b\big)^{1\slash 16}\,\Pi_{\a+e_{1}}(ds)\\
\equiv&|x-x'|\big(h_{3}(x,y,\z)+h_{4}(x,y,\z)\big).
\end{align*}
It this position, to prove the smoothness bound it suffices to show that
\begin{align*}
\|h_{k}(x,y,\z(t))\|_{L^2(dt)}
\lesssim
\frac{1}{|x-y|} \;
\frac{1}{\mu_{\alpha}(B(x,|y-x|))},\qquad k=3,4.
\end{align*}
This, however, was done implicitly earlier, see \eqref{nier5} with the succeeding estimates and \eqref{nier6} together with the succeeding estimates, respectively.

The proof will be completed once we verify the remaining smoothness bound
$$
\big\|\delta_{1,x}^{*}\widetilde G_{t}^{\alpha,1}(x,y)-\delta_{1,x}^{*}\widetilde G_{t}^{\alpha,1}(x,y')\big\|_{L^2(dt)}\lesssim
\frac{|y-y'|}{|x-y|} \;
\frac{1}{\mu_{\alpha}(B(x,|y-x|))},\qquad |x-y|>2|y-y'|.
$$
Define (see \eqref{row})
$$
\phi(x,y,\z)=\delta_{1,x}^{*}\widetilde G_{t}^{\alpha,1}(x,y)\slash y_{1}.
$$
By the mean value theorem
\begin{align*}
\big|\delta_{1,x}^{*}\widetilde G_{t}^{\alpha,1}(x,y)-\delta_{1,x}^{*}\widetilde G_{t}^{\alpha,1}(x,y')\big|
\leq&
|y-y'||\phi(x,y',\z)|+
y_{1}\big|
\nabla_{\!y}\phi(x,\t,\z)\big|,
\end{align*}
where $\t$ is a convex combination of $y$ and $y'$. We will treat separately each of the two terms appearing on the right-hand side above.
By \eqref{row} and in view of \eqref{nier2} and Lemma \ref{lemat} we have, for $|x-y|>2|y-y'|$,
\begin{align*}
|\phi(x,y',\z)|
\lesssim&
\Big(\frac{1-\z^2}{2\z}\Big)^{d+|\a|+1}\int\e\ddd\,\Pi_{\a+e_{1}}(ds)\\
&+
x_{1}\Big(\frac{1-\z^2}{2\z}\Big)^{d+|\a|+1}
\,\z^{-1\slash 2}\int\big[\e\ddd\big]^{1\slash 2}\,\Pi_{\a+e_{1}}(ds)\\
\leq&
\sqrt{1-\z^2}\,\z^{-d-|\a|-1}\int\big[\e\bbb\big]^{1\slash 4}\,\Pi_{\a+e_{1}}(ds)\\
&+
(x_{1}+y_{1})\sqrt{1-\z^2}\,\z^{-d-|\a|-3\slash 2}
\int\big[\e\bbb\big]^{1\slash 8}\,\Pi_{\a+e_{1}}(ds)\\
\equiv&
\,h_{5}(x,y,\z)+h_{3}(x,y,\z).
\end{align*}
Next, we claim that
\begin{align*}
\|h_{5}(x,y,\z(t))\|_{L^2(dt)}
\lesssim
\frac{1}{|x-y|} \;
\frac{1}{\mu_{\alpha}(B(x,|y-x|))};
\end{align*}
since the same bound for $h_{3}(x,y,\z)$ was obtained earlier, this will imply the relevant bound for $|y-y'||\phi(x,y',\z)|$.
Changing the variable as in \eqref{zz} and then using the Minkowski inequality, Lemma \ref{oq} (taken with $b=2d+2|\alpha|$, $c=1\slash 8$, $A=\zeta^{-1}$) and Lemma \ref{lem5.4} (with $a=2$ and $T=\frac{1}{16}q_{+}$) we get
\begin{align*}
\|h_{5}(x,y,\z(t))\|_{L^2(dt)}
=&
\bigg(\int_{0}^{1}\Big(\frac{1}{\z}\Big)^{2d+2|\a|+2}
\bigg(\int\big(\e\b\big)^{1\slash 4}
\,\Pi_{\a+e_{1}}(ds)\bigg)^2\,d\z\bigg)^{1\slash 2}\\
\leq&
\int\bigg(\int_{0}^{1}\Big(\frac{1}{\z}\Big)^{2d+2|\a|}
\Big(\frac{1}{\z}\Big)^{2}\big(\e\b\big)^{1\slash 2}\,d\z\bigg)^{1\slash 2}\,\Pi_{\a+e_{1}}(ds)\\
\lesssim&
\int(q_{+})^{-d-|\a|}
\bigg(\int_{0}^{1}\Big(\frac{1}{\z}\Big)^{2}\big(\e\b\big)^{1\slash 4}
\,d\z\bigg)^{1\slash 2}\,\Pi_{\a+e_{1}}(ds)\\
\lesssim&
\int(q_{+})^{-d-|\a|-1\slash 2}\,\Pi_{\a+e_{1}}(ds).
\end{align*}
Now the claim follows by applying Lemma \ref{lem4} (with $\delta=0$ and $\kappa=e_{1}$).
It remains to estimate $y_{1}\big|
\nabla_{\!y}\phi(x,\t,\z)\big|$, and to do that we first analyze $\partial_{y_{i}}\phi(x,y,\z)$.
We have (see \eqref{row})
$$
\partial_{y_{i}}\phi(x,y,\z)=
\frac{1-\z}{1+\z}\Big(\frac{1-\z^2}{2\z}\Big)^{d+|\a|+1}\partial_{y_{i}}h(x,y,\z).
$$
While treating $\partial_{y_{i}}h(x,y,\z)$ we distinguish two cases.
\newline
{\bf Case 1:} $\mathbf {i\ne 1.}$
Without any loss of generality we may restrict to $i=2$. Then
\begin{align*}
\partial_{y_{2}}&h(x,y,\z)\\
=&
(2\a_{1}+2)\int\e\b\Big[\frac{1}{2\z}(y_{2}+x_{2}s_{2})+\frac{\z}{2}(y_{2}-x_{2}s_{2})\Big]
\,\Pi_{\a+e_{1}}(ds)\\
&-
x_{1}\int\e\b\Big[\frac{1}{2\z}(y_{2}+x_{2}s_{2})+\frac{\z}{2}(y_{2}-x_{2}s_{2})\Big]
\Big[\frac{1}{2\z}(x_{1}+y_{1}s_{1})+\frac{\z}{2}(x_{1}-y_{1}s_{1})\Big]\,\Pi_{\a+e_{1}}(ds)\\
&-
x_{1}^{2}\int\e\b\Big[\frac{1}{2\z}(y_{2}+x_{2}s_{2})+\frac{\z}{2}(y_{2}-x_{2}s_{2})\Big]
\,\Pi_{\a+e_{1}}(ds).
\end{align*}
Proceeding as before (see the estimate of $\partial_{x_{2}}\v_{2}(x,y,\z)$ above) we get
\begin{align*}
|\partial_{y_{2}}h(x,y,\z)|\lesssim&
\,\z^{-1\slash 2}\int\big(\e(\z,q_{\pm})\big)^{1\slash 2}\,\Pi_{\a+e_{1}}(ds)
+
x_{1}\z^{-1}\int\big(\e(\z,q_{\pm})\big)^{1\slash 4}\,\Pi_{\a+e_{1}}(ds).
\end{align*}
{\bf Case 2:} $\mathbf {i=1.}$
An elementary computation produces
\begin{align*}
\partial_{y_{1}}&h(x,y,\z)\\
=&
(2\a_{1}+2)\int\e\b\Big[\frac{1}{2\z}(y_{1}+x_{1}s_{1})+\frac{\z}{2}(y_{1}-x_{1}s_{1})\Big]
\,\Pi_{\a+e_{1}}(ds)\\
&-
x_{1}\int\e\b\Big[\frac{1}{2\z}(y_{1}+x_{1}s_{1})+\frac{\z}{2}(y_{1}-x_{1}s_{1})\Big]
\Big[\frac{1}{2\z}(x_{1}+y_{1}s_{1})+\frac{\z}{2}(x_{1}-y_{1}s_{1})\Big]
\,\Pi_{\a+e_{1}}(ds)\\
&+
x_{1}\int\e\b\Big[\frac{1}{2\z}s_{1}-\frac{\z}{2}s_{1}\Big]\,\Pi_{\a+e_{1}}(ds)\\
&-
x_{1}^{2}\int\e\b\Big[\frac{1}{2\z}(y_{1}+x_{1}s_{1})+\frac{\z}{2}(y_{1}-x_{1}s_{1})\Big]
\,\Pi_{\a+e_{1}}(ds).
\end{align*}
Parallel arguments to those used in Case 1 lead to the same bound as before,
\begin{align*}
|\partial_{y_{1}}h(x,y,\z)|\lesssim&
\,\z^{-1\slash 2}\int\big(\e\b\big)^{1\slash 2}\,\Pi_{\a+e_{1}}(ds)
+
x_{1}\z^{-1}\int\big(\e\b\big)^{1\slash 4}\,\Pi_{\a+e_{1}}(ds).
\end{align*}
Combining the above estimates of $\partial_{y_{i}}h(x,y,\z)$ with Lemma \ref{lemat} we obtain, for $|x-y|>2|y-y'|$,
\begin{align*}
y_{1}\big|\partial_{y_{i}}\phi(x,\t,\z)\big|
\lesssim&
(x_{1}+y_{1})\sqrt{1-\z^2}\,\z^{-d-|\a|-3\slash 2}
\int\big[\e\bbb\big]^{1\slash 8}\,\Pi_{\a+e_{1}}(ds)\\
&+
x_{1}y_{1}\sqrt{1-\z^2}\,\z^{-d-|\a|-2}
\int\big[\e\bbb\big]^{1\slash 16}\,\Pi_{\a+e_{1}}(ds)\\
=&
\,h_{3}(x,y,\z)+h_{4}(x,y,\z).
\end{align*}
Finally, the estimates of $h_{3}(x,y,\z)$ and $h_{4}(x,y,\z)$ obtained earlier in this proof lead to the required bound for $y_{1}\big|\nabla_{\!y}\phi(x,\t,\z)\big|$. Now both smoothness conditions for $\big\{\delta_{1,x}^{*}\widetilde G_{t}^{\a,1}(x,y)\big\}$ are justified.

The proof of the case of $g_{H,\widetilde T}^{j,j}$ in Theorem \ref{preg} is complete.
\end {proof}
\subsection{$g$-functions based on  $\mathbf{\{\widetilde P_{t}^{\alpha,j}\}}$}\label{k6}
$\newline$
Proving Theorem \ref{preg} in the cases of $g_{V,\widetilde P}^{j}$ and $g_{H,\widetilde P}^{j,i}$ relies on the subordination principle and the kernel estimates already obtained in Sections \ref{k4} and \ref{k5}. The details are completely analogous to those in Section \ref{k3} and thus omitted.


\end{document}